\documentclass[11pt, a4paper]{article}
\usepackage[svgnames]{xcolor}
\usepackage{setspace}
\usepackage{abstract}
\usepackage[normalem]{ulem}
\usepackage{mathtools}
\usepackage[abbrev,lite,nobysame]{amsrefs}
\usepackage{amsthm,amsfonts}
\usepackage{csquotes}
\usepackage{enumerate}
\usepackage[colorlinks=true,linkcolor=NavyBlue,
citecolor=NavyBlue,urlcolor=NavyBlue,bookmarksdepth=3]{hyperref}
\usepackage{cleveref}
\usepackage{tocloft}
\usepackage{indentfirst}
\makeatletter
\newcommand*{\toccontents}{\@starttoc{toc}}
\makeatother

\crefname{equation}{}{}

\numberwithin{equation}{section}

\usepackage[looser]{newtxtext}
\usepackage[bigdelims, varvw]{newtxmath}
\usepackage[protrusion=true, tracking=true, kerning=true,spacing=true]{microtype}

\theoremstyle{plain}
\newtheorem{theorem}{Theorem}[section]
\newtheorem{lemma}[theorem]{Lemma}
\newtheorem{corollary}[theorem]{Corollary}
\newtheorem{proposition}[theorem]{Proposition}
\theoremstyle{definition}
\newtheorem{definition}[theorem]{Definition}
\theoremstyle{remark}
\newtheorem{remark}[theorem]{Remark}
\newtheorem{example}[theorem]{Example}

\renewcommand{\leq}{\leqslant}

\renewcommand{\geq}{\geqslant}
\renewcommand{\omega}{\omegaup}

\newcommand{\diff}[1]{\mathrm{d}#1}

\renewcommand{\longrightarrow}{\to}

\begin{document}\linespread{1.05}\selectfont
	\date{}
	
	\author{Kai Koike\,\protect\footnote{Department of Mathematics, Institute of Science Tokyo, 2-12-1, Ookayama, Meguro-ku, Tokyo 152-8551, Japan, e-mail: \href{mailto:koike.k@math.titech.ac.jp}{koike.k@math.titech.ac.jp}}\and Vahagn~Nersesyan\,\protect\footnote{NYU-ECNU Institute of Mathematical Sciences at NYU Shanghai, 3663 Zhongshan Road North, Shanghai, 200062, China, e-mail: \href{mailto:vahagn.nersesyan@nyu.edu}{Vahagn.Nersesyan@nyu.edu}}\and
		Manuel~Rissel\,\protect\footnote{NYU-ECNU Institute of Mathematical Sciences at NYU Shanghai, 3663 Zhongshan Road North, Shanghai, 200062, China, e-mail: \href{mailto:Manuel.Rissel@nyu.edu}{Manuel.Rissel@nyu.edu}}\and
		Marius~Tucsnak\,\protect\footnote{Institut de Math\'ematiques de Bordeaux, UMR 5251, Universit\'{e} de Bordeaux/Bordeaux INP/CNRS, 351 Cours de la Lib\'eration - F 33 405 TALENCE, France,
			and Institut Universitaire de France (IUF), e-mail: \href{mailto:Marius.Tucsnak@u-bordeaux.fr}{Marius.Tucsnak@u-bordeaux.fr}}}
	
	\title{Relaxation enhancement by \\ controlling incompressible fluid flows 
	}
	
	\maketitle

	\begin{abstract}
		We propose a PDE-controllability based approach to the enhancement of diffusive mixing for passive scalar fields. Unlike in the existing literature, our relaxation enhancing fields are not prescribed
		{\it ab initio} at every time and at every point of the spatial domain. Instead, we prove that time-dependent relaxation enhancing vector fields can be obtained
		as {\it state trajectories of control systems described by the  incompressible Euler equations} either driven by finite-dimensional controls or by controls localized in space. The main ingredient of our proof is a new approximate controllability theorem for the incompressible Euler equations on~$\mathbb{T}^2$, ensuring the approximate tracking of the full state all over the considered time interval. Combining this with a continuous dependence result  yields enhanced relaxation for the passive scalar field. Another essential tool in our analysis is the exact controllability of the incompressible Euler system driven by spatially localized forces.
	\end{abstract}
	
	\begin{center}
		\textbf{Keywords} \\ Relaxation enhancement, incompressible Euler equations, approximate tracking controllability, finite-dimensional controls, spatially localized controls
		
		\vspace{8pt}
		
		\textbf{MSC2020} \\ 35Q35, 76F25, 93B05
	\end{center}
	
	\newpage
	
	\setcounter{tocdepth}{2}
	\toccontents \vspace{4pt}
	
	\section{Introduction}\label{section:introduction}
	The objective of this work is to obtain relaxation enhancing fields as state trajectories of a system describing a controlled incompressible flow. Relaxation enhancing flows, both stationary and time-periodic, have been fully characterized by Constantin et al. and Kiselev et al.~in \cites{constantin2008diffusion,KiselevShterenbergZlatos2008}, using spectral conditions. However, these flows generally fail to satisfy fluid equations such as  the incompressible Euler or Navier--Stokes systems. Here, we offer an extended notion of relaxation enhancing flows that solve actual fluid PDEs driven by forces (controls) localized either in frequency or in the physical space. To describe the general context of our results, we consider, as in \cites{constantin2008diffusion,KiselevShterenbergZlatos2008}, the passive scalar equation of diffusion--advection type
	\begin{align}
		\frac{\partial \phi_a}{\partial t} (t,x) + a\tilde v(x) \cdot \nabla\phi_a (t,x)  - \Delta \phi_a  (t,x) = 0 && (t\geqslant 0,x\in M), \label{prima_advectie} \\
		\phi_a(0,x) = f (x) && (x\in M), \label{CI_prima_prima}
	\end{align}
	where $M$ is a smooth compact $d$-dimensional Riemannian manifold without boundary, $\Delta$ is the Laplace--Beltrami operator on $M$, $a$ is a positive constant, and $\tilde v$ is a time-independent sufficiently regular divergence-free vector field. 
	
	Under the above assumptions, the average $\overline{\phi_a}$ of $\phi_a$ on $M$ is constant with respect to time, and the $L^2(M)$ norm of $\phi_a(t,\cdot)-\overline{\phi_a}$ decays exponentially when $t\to \infty$. More precisely, we have the decay estimate
	\begin{align}\label{to_be_improved}
		\|\phi_a(t,\cdot)-\overline{\phi_a}\|_{L^2(M)}\leqslant \exp(-\lambda t) \|f-\overline f\|_{L^2(M)} && (a\geqslant 0,t\geqslant 0),
	\end{align}
	where $\lambda$ is the smallest positive eigenvalue of the operator $-\Delta$. 
	For $a=0$ the decay estimate \eqref{to_be_improved} is clearly sharp.
	Roughly speaking, the divergence-free vector field $\tilde v$ is said to be relaxation enhancing if the decay
	of the $L^2$ norm of $\phi_a$ becomes arbitrarily large when $a\to \infty$.
	We borrow from~\cite{constantin2008diffusion} the following   definition of this concept (see also the review by Coti Zelati et al.~\cite{ZelatiCrippaIyerMazzucato2024} and the references therein for further information):
	
	\begin{definition}\label{def_enhancement}
		A divergence-free vector field $\tilde v$ on $M$ is called  {\em relaxation enhancing} if for every $\tau,\delta>0$ there exists $a^*(\tau,\delta)>0$ such that for any $a\geqslant a^*(\tau,\delta)$ and every $f\in L^2(M)$, with $\|f\|_{L^2(M)}\leqslant 1$ and $\int_M f(x)\, \diff{x}=0$, the solution $\phi_a$ of \eqref{prima_advectie} and \eqref{CI_prima_prima} satisfies
		\begin{equation*}
			\|\phi_a (\tau,\cdot)\|_{L^2(M)}<\delta.
		\end{equation*}
	\end{definition}
	
	Let us mention that the above defined notion is strongly related to the concept of mixing vector fields (for instance, see \cite{ZelatiDelgadinoElgindi2020}).
	
	The main result of Constantin et al.~\cite{constantin2008diffusion} gives a necessary and sufficient condition for a Lipschitz field $\tilde v$ to be relaxation enhancing. The Lipschitz property assumption
	has been weakened to continuity by Wei~\cite{wei2021diffusion}, yielding the following~result:
	
	\begin{theorem}\label{th_wei_1}
		A continuous divergence-free field $\tilde v$ on $M$ is relaxation enhancing if and only if the operator $\varphi\mapsto \tilde v\cdot \nabla \varphi$ has no  eigenvectors in the Sobolev space $H^1(M)$, other than constant functions.
	\end{theorem}
	
	The main novelty brought in by the present work is that, instead of directly imposing
	a relaxation enhancing field $\tilde{v}$ in \eqref{prima_advectie}, 
	we obtain a (time-dependent) relaxation enhancing field as state trajectory of a system describing a controlled fluid flow on $M$. This means, in particular, that enhanced relaxation is achieved either by acting on a finite number of scalar control functions or by a control force localized in an  open strict subset of $M$.
	More precisely, we consider the case in which $M$ is the two dimensional flat torus $\mathbb{T}^2=\mathbb{R}^2/2\pi \mathbb{Z}^2$ and the
	control system is described by the incompressible Euler  equations on $\mathbb{T}^2$, i.e.,
	\begin{align}
		\frac{\partial v}{\partial t}+(v\cdot\nabla)v+
		\nabla p=h+Bu && \text{in $(0,\infty) \times \mathbb{T}^2$}, \label{prima_curgere}\\
		\operatorname{div}{v}=0 && \text{in $(0,\infty) \times \mathbb{T}^2$}, \label{divergenta_zero} \\
		v(0, \cdot) = v_0 && \text{on $\mathbb{T}^2$}, \label{CI_prima_curgere}
	\end{align}
	where $v_0$ is a divergence-free vector field and $h$ is a given function (both smooth enough in a sense made precise later).
	Here, $B\in \mathcal{L}(U,L^2(\mathbb{T}^2;\mathbb{R}^2))$ is the {\em control~operator}, $U$ is a Hilbert space, and $u\in L^2_{loc}((0,\infty);U)$ is the {\em control~function}.
	
	As far as we know, this approach has not yet been explored in the literature in the context of enhanced dissipation.
	However, the idea of replacing an apriori given divergence-free vector field $\tilde v$ in \eqref{prima_advectie} by a velocity field obtained as a state trajectory of a controlled incompressible flow already appeared in the related 
	context of mixing problems; see Hu~\cite{hu2020approximating} or  Hu et al.~\cite{hu2023feedback}. These works consider optimal control problems in which the cost function involves a negative Sobolev norm of the final state of an advection equation (with no diffusion term) and the
	advection velocity field is obtained by controlling a fluid flow described by the incompressible time-dependent Stokes or Navier--Stokes equations. Note that the advection velocity fields obtained by minimizing the cost functions in \cite{hu2020approximating} or \cite{hu2023feedback} have apriori no reason to be relaxation enhancing.
	
	In addition, we would like to mention the work by Bedrossian et al.~\cite{Stochastic21}. In their work, instead of control theoretic methods, they considered a probabilistic approach. More precisely, they obtained (in an almost sure sense) relaxation enhancing divergence-free vector fields as solutions to the 2D Navier--Stokes equations driven by a regular-in-space and non-degenerate white-in-time noise.
	The particularity of our work is that we construct relaxation enhancing fields by means of {\em finite-dimensional} or {\em spatially localized} deterministic controls acting on the 2D Euler equations. Here, finite-dimensional controls means that the control space is $U=\mathbb{R}^m$, with $m\in \mathbb{N}$ (as small as possible) and that the control operator $B$ in \eqref{prima_curgere} is given by
	\begin{equation}\label{first_choice}
		(B{\rm u})(x) =\sum_{j=1}^m {\rm u_j} \theta_j(x) \qquad \left({\rm u}=
		[{\rm u_1}, \dots, {\rm u}_m]^{\top}\in \mathbb{R}^m, x\in \mathbb{T}^2 \right),
	\end{equation}
	where $\{\theta_j\}_{1\leqslant j\leqslant m}$ are fixed (appropriately chosen) functions. On the other hand, spatially localized controls means that the support in space of $Bu$ is contained in some proper subset of $\mathbb{T}^2$; see Theorem~\ref{th_con_local}  below fore more details.
	
	Assuming that \eqref{prima_curgere}--\eqref{CI_prima_curgere} admits a solution $v$, we consider the associated diffusion--advection system
	\begin{align}
		\frac{\partial \varphi_v}{\partial t} (t,x) +  v(t,x) \cdot \nabla\varphi_v (t,x)  - \Delta \varphi_v  (t,x) = 0 && (t\geqslant 0,x\in \mathbb{T}^2), \label{a_doua_advectie} \\
		\varphi_v(0,x) = f (x) && (x\in \mathbb{T}^2). \label{CI_prima}
	\end{align}
	
	By analogy with  Definition \ref{def_enhancement}, we set
	
	\begin{definition}\label{def_enhancement_bis}
		Let $v_0=h=0$. The control system \eqref{prima_curgere}--\eqref{CI_prima_curgere}  is said to be {\em relaxation enhancing} if for every $\tau,\delta>0$, there exists
		$u\in L^2([0,\tau];U)$ such that the solution $\varphi_v$ of the system \eqref{a_doua_advectie} and \eqref{CI_prima}, where $v$ is the solution of~\eqref{prima_curgere}--\eqref{CI_prima_curgere}, satisfies
		\begin{equation*}
			\|\varphi_v (\tau,\cdot)\|_{L^2(\mathbb{T}^2)}<\delta
		\end{equation*}
		for every $f\in L^2(\mathbb{T}^2)$ with $\| f \|_{L^2(\mathbb{T}^2)}\leqslant 1$ and $\int_{\mathbb{T}^2} f(x)\, \diff{x}=0$.
	\end{definition}
	
	\begin{remark}
		The simplifying assumption $v_0=h=0$ in the definition above will be removed in our main result; see Theorem~\ref{th_rel_new}.
	\end{remark}
	
	In other words, a relaxation enhancing control system can produce a (time-dependent)
	relaxation enhancing velocity field by means of a small number of control functions
	acting as the external force of the incompressible Euler equations.
	
	The precise statements of our main results require some preparation, so we postpone them to 
	Theorem \ref{th_rel_new} and Theorem \ref{th_con_local} in Section \ref{sec_prel}. However, we can already state here a consequence
	of Theorem \ref{th_rel_new}.
	
	\begin{corollary}\label{cor_first}
		Let $m=4$ and
		\begin{alignat*}{2}
			\theta_1(x) & = [0,1]^{\top}\sin (x_1), \quad & \theta_2(x) & = [0,1]^{\top}\cos (x_1), \\
			\theta_3(x) & = [1,1]^{\top}\sin (x_1-x_2), \quad & \theta_4(x) & = [1,1]^{\top}\cos (x_1-x_2).
		\end{alignat*}
		Then the control system \eqref{prima_curgere}--\eqref{CI_prima_curgere}, with $B$ given by \eqref{first_choice}, is relaxation enhancing.
	\end{corollary}
	
	The remaining part of this article is organized as follows. In \Cref{sec_prel}, we provide precise statements of our main results and briefly describe the strategy of the proofs. In Section \ref{sec_robust}, we show that the mapping that sends a divergence-free advection field to the solution of the corresponding advection--diffusion system is H\"older continuous from the space of advection velocities  endowed with the relaxation norm~\eqref{def_norm_relax}
	to standard function spaces for the solutions of the diffusion--advection system. \Cref{extremely_difficult} is exclusively devoted to the proof of approximate tracking controllability for the Euler system driven by a frequency-localized force.
	In Section~\ref{sec_exact_Euler}, we give a short proof for the exact controllability of the incompressible Euler equations on $\mathbb{T}^2$ with controls supported in an arbitrary open set $\omegaup \subset \mathbb{T}^2$ for which $\mathbb{T}^2 \setminus \omegaup$ is simply-connected.
	Finally, in Section \ref{sec_proofs_main}, we complete the proofs of our main results.
	
	
	\section{Main results}\label{sec_prel}
	
	In this section, we introduce some notation and preliminaries used throughout this paper. After this preparation, we state our main results in Theorems~\ref{th_rel_new} and~\ref{th_con_local}.
	
	For $d\in \mathbb{N}$ with $d\geqslant 2$, the standard inner product of two vectors $v,w\in \mathbb{R}^d$ is denoted by $v\cdot w$, while $\mathbb{T}^d$  stands for the standard $d$-dimensional torus~$\mathbb{R}^d/2\pi\mathbb{Z}^d$.
	Moreover, for every $k\in \mathbb{N}\cup\{0\}$ we denote by $H^k$ the Sobolev space of zero average scalar functions
	\begin{equation}\label{defhk_nou}
		H^k \coloneqq \left\{\varphi\in H^k(\mathbb{T}^d) \,:\, \int_{\mathbb{T}^d} \varphi(x)\, \diff{x}=0\right\}.
	\end{equation}
	The space $H^0$ is endowed with the standard $L^2$ inner product and induced norm.
	
	Let $-A_0$ be the Laplacian on $H^0$, which means that $A_0 \colon \mathcal{D}(A_0)\to H^0$ with
	\begin{equation}\label{def_dom_A0}
		\mathcal{D}(A_0)=H^2,
	\end{equation}
	\begin{equation}\label{def_op_A0}
		A_0 g= -\Delta g \qquad (g\in \mathcal{D}(A_0)).
	\end{equation}
	The result below gathers, for later use, some well-known properties of $A_0$.
	
	\begin{proposition}\label{prop_A0}
		Let $A_0$ be the operator defined in \eqref{def_dom_A0} and \eqref{def_op_A0}. Then
		\begin{enumerate}
			\item[\rm (i)] $A_0$ is strictly positive on $H^0$, which means that $A_0$ is self-adjoint 
			and there exists $c_0>0$ such that
			$$
			\langle A_0 g,g\rangle_{L^2(\mathbb{T}^d)} \geqslant c_0 \|g\|_{L^2(\mathbb{T}^d)}^2 \qquad (g\in \mathcal{D}(A_0)).
			$$
			\item[\rm (ii)] For every $m\in \mathbb{N}$ we have that $\mathcal{D}(A_0^m)=H^{2m}$. Moreover, for every $k\in \mathbb{N}$,
			the standard inner product in $H^k$ is equivalent to the inner product $\langle\cdot,\cdot\rangle_k$ defined by
			\begin{equation}\label{def_scal_l}
				\langle g_1,g_2\rangle_k =\langle A_0^{\frac{k}{2}} g_1, A_0^{\frac{k}{2}} g_2\rangle_{L^2(\mathbb{T}^d)} \qquad (g_1,g_2\in H^k).
			\end{equation}
			\item[\rm (iii)] For every $k\in \mathbb{N}$ the part of $A_0$ in $H^k$ defines a strictly positive operator on $H^k$, with domain $H^{k+2}$.
		\end{enumerate}
	\end{proposition}
	
	Next, for $d,k\in \mathbb{N}$ we denote by  $H^k(\mathbb{T}^d; \mathbb{R}^d)$ the space of vector functions $v = [v_1, \dots, v_d]^\top$ with components belonging to the Sobolev space $H^k$ defined in~\eqref{defhk_nou}. For the remaining part of this work,
	the standard $H^k$ norm is denoted by~$\|\cdot\|_k$. 
	We also set
	\begin{equation}\label{numarh}
		H^k_\sigma \coloneqq H^k_\sigma(\mathbb{T}^d;\mathbb{R}^d) \coloneqq
		H^k(\mathbb{T}^d; \mathbb{R}^d)\cap \mathcal{H},
	\end{equation}
	where~$\mathcal{H}$ is given by 
	\begin{equation}\label{DEFH}
		\mathcal{H} \coloneqq \left\{v\in L^2(\mathbb{T}^d; \mathbb{R}^d)\,:\, \operatorname{div}{v}=0,\,\, \int_{\mathbb{T}^d}v(x)\, \diff{x}=0\right\}.
	\end{equation}
	Note that $H^k_\sigma$ is a Hilbert space when endowed with the norm
	\begin{equation*}
		\|v\|_{H^k_\sigma}^2=\sum_{j=1}^d \|v_j\|_k^2.
	\end{equation*}
	
	We denote by  $\Pi$  the orthogonal projection  from $L^2(\mathbb{T}^d;\mathbb{R}^d)$ onto $\mathcal{H}$, where the space $\mathcal{H}$ has been defined in \eqref{DEFH}.
	Moreover, for $k > \frac{d}{2}$  we define
	\begin{equation}\label{nonlinear}
		N(v) =\Pi \left[(v\cdot \nabla) v\right] \qquad (v\in H^k(\mathbb{T}^d; \mathbb{R}^d)).
	\end{equation}
	We next recall some notation from Agrachev and Sarychev~\cites{AS-2006} and Shirikyan~\cite{shirikyan-cmp2006}. For any finite-dimensional subspace
	$E\subset H_\sigma^{k+2}$, we denote by $\mathcal{F}(E)$ the largest
	vector space $F\subset H^{k+2}_\sigma$ such that any $\eta_1\in F$ can be represented in the form
	$$
	\eta_1=\eta-\sum_{i=1}^{ p} N(\zeta^i)
	$$ 
	for some $p\in \mathbb{N}$ and vectors $\eta, \zeta^1,\ldots ,\zeta^p\in E$. As the space $E$ is finite-dimensional and $N$ is quadratic, it is easy
	to see that~$\mathcal{F}(E)$ is well-defined. A non-decreasing sequence of subspaces $\{E_j\}_{j\in \mathbb{N}\cup \{ 0,\infty \}}$ is defined by recurrence as follows:
	\begin{equation}\label{spatii_stranii}
		E_0=E,\quad E_j=\mathcal{F}(E_{j-1})\quad \textrm{for $j \in \mathbb{N}$},\quad E_\infty=\bigcup_{j=1}^\infty E_j. 
	\end{equation}

	\begin{definition}
		The space $E$ is said to be {\it saturating} if $E_\infty$ is dense in $H^{k}_\sigma$.
	\end{definition}

	\begin{remark}
		We note that for the saturating space $E$ given in Example~\ref{example:saturating} (cf.~Lemma~\ref{lem:HM}), the set $E_{\infty}$ is actually dense in $H_{\sigma}^{l}$ for all $l\in \mathbb{N}$.
	\end{remark}
	
	We are now in a position to state the first main result of this paper:
	
	\begin{theorem}\label{th_rel_new}
			Let $v_0 \in H_{\sigma}^{4}$ and $h\in L_{loc}^{1}([0,\infty);H_{\sigma}^{4})$. Assume that $E$ is a finite-dimensional subspace of $H_\sigma^{6}$ and that it is saturating. Let $\{\theta_1,\dots,\theta_m\}$ be a basis of $E$ and define $B$ by \eqref{first_choice}.
		Then for every $\tau,\delta>0$, there exists $u\in C^{\infty}([0,\tau];\mathbb{R}^m)$ such that the solution $\varphi_v$ of the system \eqref{a_doua_advectie} and \eqref{CI_prima}, where $v$ is the solution of~\eqref{prima_curgere}--\eqref{CI_prima_curgere}, satisfies
		\begin{equation}\label{pana_aici_2}
			\|\varphi_v (\tau,\cdot)\|_{L^2(\mathbb{T}^2)}<\delta
		\end{equation}
		for every $f\in L^2(\mathbb{T}^2)$ with $\| f \|_{L^2(\mathbb{T}^2)}\leqslant 1$ and $\int_{\mathbb{T}^2} f(x)\, \diff{x}=0$. In particular, the control system \eqref{prima_curgere}--\eqref{CI_prima_curgere}, with $B$ given by \eqref{first_choice}, is relaxation enhancing.
\end{theorem}

A possible choice of a space $E$ satisfying the assumptions in \Cref{th_rel_new} is given by
\Cref{example:saturating}.
Then, Corollary~\ref{cor_first} is just Theorem~\ref{th_rel_new} with this choice of~$E$.

In \Cref{th_rel_new},
the input function takes values in a finite-dimensional space $E$ and is active on all of $\mathbb{T}^2$.
An interesting question is whether one can construct relaxation enhancing control systems with inputs   which are active only on a strict subset of $\mathbb{T}^2$ (also called 
{\em spatially localized controls}).
At this stage, however, it is difficult to achieve the conclusion of Theorem~\ref{th_rel_new} with such controls. Let us explain the difficulty: to cause relaxation enhancement, we need to make sure by choosing an appropriate control input $u$ in~\eqref{prima_curgere}--\eqref{CI_prima_curgere} that the solution $v$ stays close to $a\tilde{v}$ (where $\tilde{v}$ is relaxation enhancing in the sense of Definition~\ref{def_enhancement} and $a$ is a large constant) during a time interval. However, controlling solutions to PDEs over an interval of time---as opposed to controlling at an instant of time---is in general a very difficult task. To our knowledge, no such result with spatially localized control is known for nonlinear fluid dynamical equations (see the discussion of \textit{tracking controllability} in the beginning of Section~\ref{extremely_difficult}).

Nevertheless, we give below a partial result for relaxation enhancement via distributed internal controls. Namely, we prove that it is possible to enhance the decay of a certain projection of the solution $\varphi_v$ of~\eqref{a_doua_advectie} and~\eqref{CI_prima}. More precisely, we have: 

\begin{theorem}\label{th_con_local}
	Let $\omegaup \subset \mathbb{T}^2$ be open and such that $\mathbb{T}^2\setminus \omegaup$ is simply-connected. 
	Let $B\in \mathcal{L}(L^2(\mathbb{T}^2))$ be defined by 
	\begin{equation*}
		B {\rm u} = \mathbb{I}_{\omegaup} {\rm u} \qquad ({\rm u}\in L^2(\mathbb{T}^2)),
	\end{equation*}
	where $\mathbb{I}_{\omegaup}$ is the characteristic function of $\omegaup$.
	Moreover, let $X$ be the Hilbert space
	\begin{equation}\label{def_space_x}
		X=\left\{ g\in L^2(\mathbb{T}^2) \, :\, \int_{\mathbb{T}} g(x_1,x_2)\, \diff{x}_1=0\  \hbox{\rm for a.e.}\ x_2\in \mathbb{T} \right\},
	\end{equation}
	endowed with the $L^2(\mathbb{T}^2)$ inner product and let $\mathrm{P}_X$ be the orthogonal projector from $L^2(\mathbb{T}^2)$ onto $X$. 
	Then for every $\tau,\delta>0$ there exists 
	$u\in L^2([0,\tau];L^2(\mathbb{T}^2))$ such that the solution $\varphi_v$ of the system \eqref{a_doua_advectie} and \eqref{CI_prima}, where $v$ satisfies \eqref{prima_curgere}--\eqref{CI_prima_curgere}  with $h=0$, is such that
	\begin{equation*}
		\|\mathrm{P}_X\varphi_v (\tau,\cdot)\|_{L^2(\mathbb{T}^2)}<\delta  ,
	\end{equation*}
	for every $f\in L^2(\mathbb{T}^2)$ with $\|f\|_{L^2(\mathbb{T}^2)}\leqslant 1$.
\end{theorem} 

\begin{remark}
	An example of $\omega$ satisfying the assumption in Theorem~\ref{th_con_local} is
	$$
	\omega=\{ (a,b)\times (0,2\pi) \} \cup \{ (0,2\pi)\times (c,d) \} \subset \mathbb{T}^2,
	$$
	where $0<a<b<2\pi$ and $0<c<d<2\pi$.
\end{remark}

Let us also remark that by choosing an appropriate external force $h$ so that a relaxation enhancing vector field $\tilde{v}$ is an exact solution of~\eqref{prima_curgere}--\eqref{CI_prima_curgere}, we can get rid of the projection. More precisely, we have:

\begin{proposition}\label{prop24}
	Assume that $\tilde{v}\in H^{6}(\mathbb{T}^2;\mathbb{R}^2)$ is a relaxation enhancing divergence-free vector field in the sense of Definition~\ref{def_enhancement}.
	Let $\tau,\delta>0$, and let $a>0$ be such that for every $f\in L^2(\mathbb{T}^2)$, with $\|f\|_{L^2(\mathbb{T}^2)}\leqslant 1$ and $\int_{\mathbb{T}^2} f(x)\, \diff{x}=0$, the solution $\phi_a$ of \eqref{prima_advectie} and \eqref{CI_prima_prima} satisfies
	\begin{equation*}
		\|\phi_a (\tau,\cdot)\|_{L^2(\mathbb{T}^2)}<\delta.
	\end{equation*}
	Moreover, let $h\in H^5(\mathbb{T}^2;\mathbb{R}^2)$ be such that
	$$
	\nabla \wedge \left[h-a^2(\tilde v\cdot \nabla\tilde v)\right]=0.
	$$
	Then under the notation and the assumptions in Theorem \ref{th_con_local},
	there exists $u\in L^2([0,\tau];H^2(\mathbb{T}^2;\mathbb{R}^2))$ such that
	the solution $\varphi_v$ of the system \eqref{a_doua_advectie} and \eqref{CI_prima} satisfies~\eqref{pana_aici_2}, where $v$ is the solution of~\eqref{prima_curgere}--\eqref{CI_prima_curgere} with $h$ as above.
\end{proposition}

\begin{remark}
	A natural open question would be to remove the projection in Theorem~\ref{th_con_local}.~An even harder problem is to consider controls active on a strict subset of $\mathbb{T}^2$ \textit{and} taking values in a finite-dimensional space $E$. This would be too ambitious for the present study even with the projection, as already the basic question of final state approximate controllability for incompressible fluids driven by finite-dimensional and physically localized controls corresponds to a well-known open problem posed by Agrachev in \cite{Agrachev2014}*{Section 7}; see also~Nersesyan and Rissel~\cite{NersesyanRissel2024}.
\end{remark}

To end this section we describe, for reader's convenience, the main ideas in the proofs of our main results in Theorem \ref{th_rel_new} and Theorem~\ref{th_con_local}. 

The main ingredient of the proof of  Theorem \ref{th_rel_new} is a new {\it approximate  tracking controllability} result for a system described by the incompressible Euler equations with a finite number of scalar controls. The term approximate tracking controllability means that the state of the system is maintained close (in an appropriate sense) to a reference trajectory for every time smaller than some given $T>0$. This result is then used to prove Theorem \ref{th_rel_new} in two steps. Firstly, we use the approximate tracking controllability result
to obtain a control $u$ such that the state trajectory $v$ of \eqref{prima_curgere}--\eqref{CI_prima_curgere}
(with $B$ given by \eqref{first_choice}) stays close to $a\tilde v$, where $a$ and $\tilde{v}$ are chosen so that the solution $\phi_a$ of~\eqref{prima_advectie} and~\eqref{CI_prima_prima} satisfies \eqref{decadere_mare_dubla} for any given $\tau,\delta>0$. We next show, using Theorem \ref{L:1} below, that the corresponding solution $\varphi_v$ of 
\eqref{a_doua_advectie} and \eqref{CI_prima} stays close to $\phi_a$; thus, it satisfies~\eqref{pana_aici_2}.

We use a different strategy to prove Theorem \ref{th_con_local}. 
The main idea is that the enhancement of decay of the projection is possible with \textit{shear flows} (see e.g. Bedrossian and Coti Zelati~\cite{bedrossian2017enhanced}). Since shear flows are exact solutions to the Euler system, controllability of solutions to a shear flow at an instant of time implies \textit{exact tracking controllability} of the shear flow after this instant: we just turn off the control immediately after the solution of the Euler system reached the shear flow, then the solution remains the same shear flow afterwards. Thus, Theorem~\ref{th_con_local} can be proved using a global exact controllability result for the Euler equation, which we develop in Section~\ref{sec_exact_Euler}. We note that this strategy does not work for the Navier--Stokes equations since there is no nontrivial periodic stationary shear flow solution. On the other hand, it is possible to extend Theorem~\ref{th_rel_new} to the Navier--Stokes equations, where the analysis would be less technical due to better regularity properties.

\section{A continuous dependence result}\label{sec_robust}

The aim of this section is to prove that the solution $\varphi_v$ of \eqref{a_doua_advectie} and \eqref{CI_prima} depends continuously
on the initial data $f$ and the advection velocity $v$. While continuity with respect to $f$ involves the norms of the standard Sobolev spaces defined 
in \eqref{defhk_nou}, the continuity with respect to $v$ involves a less standard norm, which is defined in \eqref{def_norm_relax} below. As the results in this section are not specific to two  space dimensions, we  consider here the following diffusion--advection system in $\mathbb{T}^d$ with a fixed integer $d\geqslant 2$:
\begin{align}
	\frac{\partial \varphi_v}{\partial t} (t,x) +  v(t,x) \cdot \nabla\varphi_v (t,x)  - \Delta \varphi_v  (t,x) = 0 && (t\geqslant 0,x\in \mathbb{T}^d), \tag{\ref{a_doua_advectie}$'$} \label{a_doua_advectie_prime} \\
	\varphi_v(0,x) = f (x) && (x\in \mathbb{T}^d), \tag{\ref{CI_prima}$'$} \label{CI_prima_prime}
\end{align}
where $v$ is a given divergence-free vector field on $\mathbb{T}^d$.

Let $X$ be a Banach space endowed with a
norm denoted by $\|\cdot\|_X$, $T>0$, and $J_T=[0,T]$. For $1\leqslant p<\infty$, the notation $L^p(J_T;X)$ stands for the
space of strongly measurable functions $v \colon J_T \rightarrow X$ such that
\begin{equation}
	\|v\|_{L^p(J_T;X)} \coloneqq \bigg(\int_0^T \|v(s)\|_X^p\, \diff{s}
	\bigg)^{\frac{1}{p}}<\infty.\nonumber
\end{equation}
The spaces $C(J_T;X)$ and $W^{k,p}(J_T;X)$ are defined in a similar way.   
We first note that the system \eqref{a_doua_advectie_prime} and \eqref{CI_prima_prime} is well-posed.
Although versions of this result are currently used in the literature we provide, for readers' convenience, a complete proof below.

\begin{proposition}\label{th_exist_unique}
	Let $T>0$ and $k\in \mathbb{N}$ with $k>\frac{d}{2}$. Then for every $f\in  H^{k+2}$ and $v \in C(J_T;H^{k+2}_\sigma)$
	there exists a unique solution 
	\begin{equation}\label{very_explicit}
		\varphi_v \in C(J_T;H^{k+2})\cap L^2(J_T;H^{k+3})\cap H^1(J_T;H^{k+1})
	\end{equation}
	of \eqref{a_doua_advectie_prime} and \eqref{CI_prima_prime}. Moreover, assume that
	$f$ and $v$ are such that 
	\begin{equation*}
		\|f\|_{k+2}+\|v\|_{C(J_T;H^{k+2}_\sigma)}\leqslant R
	\end{equation*}
	for some $R>0$. Then there exists a constant $C \coloneqq C(R,T,k,d)>0$ such that 
	\begin{equation}\label{marginire_c}
		\|\varphi_v\|_{C(J_T;H^{k+2})}+\|\varphi_v\|_{L^2(J_T;H^{k+3})} \leqslant C.
	\end{equation}
\end{proposition}

\begin{proof}
	Let $\tilde V=H^{k+3}$ and $\tilde H=H^{k+2}$, where the spaces $H^l$, with $l\in \mathbb{N}$, have been introduced
	in \eqref{defhk_nou} and they are endowed with the inner products defined in \eqref{def_scal_l}.
	It is easily seen that $\tilde V\subset \tilde H$ with continuous and dense embedding and that the dual ${\tilde V}'$ of $\tilde V$ with respect to the pivot space $\tilde H$ is
	$H^{k+1}$.
	
	For every $t\in [0,T]$  we define the linear operator $L(t)$ mapping  $\gamma\in  \tilde{V}$ into $v(t,\cdot)\cdot \nabla \gamma$.
	Using the fact that $\tilde{H}$ is a Banach algebra,  we have
	\begin{align}
		\|L(t) \gamma\|_{{\tilde H}} \notag
		& =\|L(t) \gamma\|_{k+2} \leqslant \sum_{j=1}^d
		\left\|v_j(t,\cdot)\frac{\partial\gamma}{\partial x_j}\right\|_{k+2} \notag \\
		& \leqslant \sum_{j=1}^d\|v_j(t,\cdot)\|_{k+2}\, \left\|\frac{\partial\gamma}{\partial x_j}\right\|_{k+2} \notag \\
		& \leqslant d C M_T \|\gamma\|_{\tilde V} \qquad (\gamma\in \tilde V,\ t\in J_T\ {\rm a.e.}), \label{very_important}
	\end{align}
	where $C=C(k,d)>0$ is the constant for the Sobolev inequality and
	\[
	M_T=\max_{j\in \{1,\dots,d\}}\, \max_{t\in [0,T]}\,  \|v_j(t,\cdot)\|_{{k+2}}.
	\]
	It follows that the bilinear form
	\begin{align}
		a(t;\gamma,\zeta)
		& =\langle \gamma,\zeta\rangle_{\tilde V}+\langle L(t)\gamma,\zeta\rangle_{\tilde H} \notag \\
		& =\langle \gamma,\zeta\rangle_{k+3}  +\langle L(t)\gamma,\zeta\rangle_{k+2} \qquad (t\in J_T, \ \gamma,\zeta\in \tilde V) \label{forma_buna}
	\end{align}
	satisfies
	\begin{equation}\label{marginire_a}
		|a(t;\gamma,\zeta)|\leqslant (1+{d}C M_T) \|\gamma\|_{\tilde V}\, \|\zeta\|_{\tilde V} \qquad (t\in J_T, \ \gamma,\zeta\in \tilde V).
	\end{equation}
	Moreover, using \eqref{very_important}, it follows that for every $\zeta\in \tilde V$ we have
	\begin{equation*}
		{\rm Re}\, \langle L(t) \gamma,\gamma\rangle_{{\tilde H}}\leqslant d C M_T \|\gamma\|_{\tilde V}\,
		\|\gamma\|_{\tilde H} \leqslant \frac12 \left(\|{\gamma}\|_{\tilde V}^2+d^2 C^{2} M_T^2  \|{\gamma}\|_{\tilde H}^2\right).
	\end{equation*}
	Putting together the above estimate and \eqref{forma_buna} it follows that
	
	\begin{equation}\label{inferiora_a}
		{\rm Re}\, a(t;\gamma,\gamma)\geqslant \frac12 \|\gamma\|_{\tilde V}^2 -
		\frac12  d^2 C^{2} M_T^2  \|\gamma\|_{\tilde H}^2 \qquad (t\in J_T, \ \gamma\in \tilde V).
	\end{equation}
	By combining \eqref{marginire_a}, \eqref{inferiora_a}, and a classical result, see for instance Dautray and Lions~\cite{Dautray_Lions}*{p.~513} or Dier and Zacher~\cite{Dier}, it follows 
	that for every $f\in \tilde H$   and  $v\in C(J_T;H^{k+2}_\sigma)$ there exists a unique function 
	\begin{equation}\label{reg_abs}
		\varphi_v \in C(J_T;\tilde H)\cap L^2(J_T;\tilde V)\cap H^1(J_T;{\tilde V}')
	\end{equation}
	satisfying 
	\begin{equation}\label{ec_forma}
		\left\langle \frac{\partial\varphi_v}{\partial t}(t,\cdot),\zeta\right\rangle_{{\tilde V}',\tilde V}+a(t;\varphi_v(t,\cdot),\zeta)=0 \qquad (\zeta\in \tilde V,\ t\in J_T\ {\rm a.e.}),
	\end{equation}
	\begin{equation}\label{ci_forma}
		\varphi_v(0,\cdot)=f.
	\end{equation}
	It is easily seen that $\varphi_v$ satisfies \eqref{reg_abs}, \eqref{ec_forma}, and \eqref{ci_forma} if and only if
	$\varphi_v$ satisfies \eqref{very_explicit}, \eqref{a_doua_advectie_prime}, and \eqref{CI_prima_prime}. Thus, for every $f\in  H^{k+2}$ and $v \in C(J_T;H^{k+2}_\sigma)$
	there exists a unique solution 
	\begin{equation*}
		\varphi_v \in C(J_T;H^{k+2})\cap L^2(J_T;H^{k+3})\cap H^1(J_T;H^{k+1})
	\end{equation*}
	of \eqref{a_doua_advectie_prime} and \eqref{CI_prima_prime}.
	
	Taking next $\zeta=\varphi_v(t,\cdot)$ in \eqref{ec_forma} and using \eqref{inferiora_a} it follows that
	$$
	\frac12 \frac{\diff{{}}}{\diff{t}}\|\varphi_v(t,\cdot)\|_{\tilde H}^2+\frac12 \|\varphi_v(t,\cdot)\|_{\tilde V}^2
	\leqslant  \frac12 d^2 C^{2} M_T^2  \|\varphi_v(t,\cdot)\|_{\tilde H}^2.
	$$
	Finally, the conclusion \eqref{marginire_c} follows from the last estimate by applying Gr\"onwall's inequality.
\end{proof}

In order to state the main result in this section we introduce the {\it relaxation norm} of $v\in L^1(J_T;X)$, where $X$ is a Banach space, defined as follows (cf.~Gamkrelidze~\cite{G-78} and~Agrachev and Sarychev~\cite{AS-2006}):
\begin{equation}\label{def_norm_relax}
	||| v |||_{T,X} \coloneqq \sup_{t\in J_T}  \left\|\int_0^t v(s)\, \diff{s}\right\|_X.
\end{equation}

\begin{theorem} \label{L:1}
	Let $T>0$ and $k\in \mathbb{N}$ with $k>\frac{d}{2}$. For every $f\in  H^{k+2}$ and $v_1,v_2 \in C(J_T;H^{k+2}_\sigma)$
	let 
	\begin{equation*}
		\varphi_{v_1},\varphi_{v_2} \in C(J_T;H^{k+2})\cap L^2(J_T;H^{k+3})\cap H^1(J_T;H^{k+1})
	\end{equation*}
	be the corresponding solutions of \eqref{a_doua_advectie_prime} and \eqref{CI_prima_prime}, as constructed in Proposition~\ref{th_exist_unique}. Moreover, assume that
	$f,v_1,v_2$ are such that 
	\begin{equation}\label{E:1.4_1_2}
		\|f\|_{k+2}+\|v_1\|_{C(J_T;H^{k+2}_\sigma)} + \|v_2\|_{C(J_T;H^{k+2}_\sigma)} \leqslant R
	\end{equation}
	for some $R>0$. Then there exists a constant $C \coloneqq C(R,T,k,d)>0$ such that 
	\begin{equation}\label{E:1.3}
		\|\varphi_{v_1}-\varphi_{v_2}\|_{C(J_T;H^k)}\leqslant C ||| v_1 - v_2|||_{T,k}^{1/2}, 
	\end{equation}
	where $|||\cdot|||_{T,k}$ stands for the relaxation norm defined in \eqref{def_norm_relax} with $X=H^k$.
\end{theorem}

\begin{proof}
	Within this proof we denote by the same letter $C$ a generic constant possibly depending on $R,T,k,$ and $d$.  Moreover, for notational simplicity, instead of $\varphi_{v_1}$ and $\varphi_{v_2}$ we simply write $\varphi_{1}$ and $\varphi_{2}$.
	
	We set $\varphi(t) \coloneqq \varphi_2 (t)-\varphi_1(t)$ and note that $\varphi$ satisfies
	\begin{align}
		\frac{\partial \varphi}{\partial t} -\Delta \varphi + v_1 \cdot \nabla \varphi+(v_2-v_1)\cdot \nabla \varphi_2=0 && \text{in $J_T \times \mathbb{T}^d$}, \label{prima_above} \\
		\varphi(0,\cdot)=0 && \text{in $\mathbb{T}^d$}. \label{adoua_above}
	\end{align}
	Taking the inner product in $H^{k}$ of \eqref{prima_above} with $\varphi$, integrating with respect to time and using \eqref{adoua_above}
	we obtain that: 
	\begin{align}
		\frac12\|\varphi(t)\|_{k}^2+\int_0^t\|\varphi\|_{k+1}^2\, \diff{s}
		& =-\int_0^t  \langle v_1 \cdot \nabla  \varphi, \varphi\rangle_{k}\, \diff{s}\nonumber \\
		& \phantom{\leqslant(} -\int_0^t\langle (v_2-v_1)\cdot \nabla \varphi_2,\varphi\rangle_{k}
		\, \diff{s}\nonumber\\
		&\eqqcolon I_1+I_2.\label{E:1.6}
	\end{align}
	As $k>\frac{d}{2}$ implies that $H^k$ is an algebra, using \eqref{E:1.4_1_2}, we have 
	\begin{align}
		|I_1|
		& \leqslant \int_0^t \|v_1\|_{k} \, \|\varphi\|_{k+1} \, \|\varphi\|_{k}\, \diff{s} \notag \\
		& \leqslant \frac12\int_0^t \|\varphi\|_{k+1}^2 \, \diff{s}+\frac12\int_0^t \|v_1\|_{k}^2 \, \|\varphi\|_{k}^2\, \diff{s} \notag \\
		& \leqslant \frac12\int_0^t \|\varphi\|_{k+1}^2 \, \diff{s}+C\int_0^t \|\varphi\|_{k}^2\, \diff{s} \label{E:1.7}
	\end{align}
	for some $C>0$. To estimate $I_2$, let $\tilde v(t)=\smallint_0^t (v_1(\sigma)-v_2(\sigma))\, \diff{\sigma}$. Then, denoting the partial derivative with respect to $t$ by $'$,
	\begin{align}\label{minunat}
		\begin{aligned}
			I_2
			& =\int_0^t\langle (v_1-v_2)(s)\cdot \nabla \varphi_2(s),\varphi(s) \rangle_{k}\, \diff{s} \\
			& =\int_0^t\langle {\tilde v}'(s)\cdot \nabla \varphi_2(s),\varphi(s)\rangle_{k}\, \diff{s} \\
			& = \langle {\tilde v} (t)\cdot \nabla \varphi_2(t),\varphi(t)\rangle_{k} \\
			& \phantom{\leqslant(} -\int_0^t\langle {\tilde v}(s)\cdot \nabla \varphi_2'(s),\varphi(s)\rangle_{k}\, \diff{s} 
			-\int_0^t\langle {\tilde v}(s)\cdot \nabla \varphi_2(s),\varphi'(s)\rangle_{k}
			\, \diff{s}.
		\end{aligned}
	\end{align}
	Using the fact that $\tilde v$ is divergence-free, it is easily checked that
	\begin{equation}\label{eq:getting_rid_of_space_time_derivatives}
		\int_0^t\langle {\tilde v}(s)\cdot \nabla \varphi_2'(s),\varphi(s)\rangle_{k} \, \diff{s}=-
		\int_0^t\langle  \varphi_2'(s) \tilde v(s), \nabla \varphi(s)\rangle_{k}
		\, \diff{s}.
	\end{equation}
	Then~\eqref{minunat} and~\eqref{eq:getting_rid_of_space_time_derivatives}, together with \eqref{E:1.4_1_2} and estimate \eqref{marginire_c} in Proposition \ref{th_exist_unique}, we obtain
	\begin{equation}\label{eq:I2}
		|I_2| \leqslant C \sup_{0\leqslant s\leqslant t}\|\tilde v(s,\cdot)\|_k\leqslant C||| \tilde{v} |||_{T,k}
	\end{equation}
	for some $C>0$. Finally, \eqref{E:1.6}, \eqref{E:1.7}, and \eqref{eq:I2} combined with Gr\"onwall's inequality imply the required inequality \eqref{E:1.3}.
\end{proof}

\section{Approximate tracking controllability}\label{extremely_difficult}

In this section, we establish
a new approximate controllability result for the Euler equations using control functions with values in a fixed finite-dimensional space (see Theorem~\ref{T.2.1}).
This will play a central role in the proof of Theorem \ref{th_rel_new}. 

Our approach  is strongly inspired by the methods introduced by  Agrachev and Sarychev \cites{AS-05, AS-2006}
for the approximate controllability of the final state for the 2D Euler and Navier--Stokes equations using finite-dimensional controls.  These results have been extended to~the~3D Navier--Stokes equations by Shirikyan \cites{shirikyan-cmp2006, shirikyan-ihp-2007} and to~the~3D Euler equations by Nersisyan \cite{nersisyan-2010}.

An essential feature of our result is that, unlike in standard controllability theory, we do not control just the final state but the complete state trajectory: we prove the existence of controls such that the state trajectory remains arbitrarily close---with respect to the relaxation norm introduced in \eqref{def_norm_relax}---to any prescribed 
space-time function  $\psi=\psi(t,x)$. Remarkably, the here exposed mechanism heavily relies on the nonlinear character of the Euler system. 
As far as we know, the only existing work addressing the approximate controllability of the complete state trajectory for a nonlinear PDE system is by Nersesyan \cite{Ners-2015}, who studied this question for the 3D Navier--Stokes equations. We show below that the methods in \cites{shirikyan-cmp2006, shirikyan-ihp-2007, nersisyan-2010,Ners-2015} can be extended to obtain a new approximate controllability result for a system described by the 2D Euler equations.

On the other hand, let us note that we are unaware of any non-trivial
linear control system for which controllability of the complete state trajectory has been shown.
However, the question of tracking a finite-dimensional output of a system has been studied in several papers.
We can mention in this sense the recent work of Zamorano and Zuazua~\cite{zamorano2024tracking} considering finite-dimensional linear systems,
where this property is called \textit{tracking controllability}. Moreover, there are several results proving tracking controllability of a finite-dimensional output of an infinite-dimensional system, such as B\'{a}rcena-Petisco and Zuazua~\cite{barcena2024tracking}, in which the tracking controllability of some boundary values of the heat equation is studied. 
Another result on tracking of a finite-dimensional output can be found in  Glass et al.~\cite{GlassKolumbanSueur}, where the authors prove a tracking boundary controllability property for trajectories of rigid bodies in a two-dimensional incompressible perfect fluid (the fluid motion is not controlled).
One of the advantages of frequency-localized controls is the possibility to control the complete state trajectory of the system, which is critical for relaxation enhancement.

\subsubsection*{Notation reminder}

We continue here to use the notation introduced in Sections \ref{sec_prel} and \ref{sec_robust}, including the function spaces $\left\{ H_\sigma^k\right \}_{k\in \mathbb{N}}$ defined in \eqref{numarh}.
Moreover, for every $T>0$ we still denote by $J_T$ the interval $[0,T]$ and we consider the spaces
$L^p(J_T;X)$,  $C(J_T;X)$, and $W^{k,p}(J_T;X)$ introduced at the beginning of Section \ref{sec_prel}. 
Finally, in this section we consider only the case of two space dimensions, i.e., we take $d=2$ in the definition of the above-mentioned spaces.

\subsubsection*{Euler system driven by finite-dimensional controls}
For $k\in \mathbb{N}$, let $E$ be a finite-dimensional subspace of $H_\sigma^{k+4}$ and let $\{\theta_1,\dots,\theta_m\}$ be a basis
of $E$. Consider an ideal incompressible fluid flow described $(0,\infty)\times\mathbb{T}^2$ by the following initial-value
problem for the Euler system: 
\begin{equation}\label{equation:euler_finitedimcontrol}
	\frac{\partial v}{\partial t}  +(v\cdot\nabla)v +\nabla p=h+\sum_{j=1}^m u_j \theta_j, \quad
	{\rm div}\, v=0, \quad
	v(0,\cdot) = v_0  
\end{equation}
where $h$ is the fixed part of the external force (given function, assumed to be smooth enough with respect to the space variable, divergence-free, and integrable with respect to time) and $u=[u_1, \dots, u_m]^\top$ is the control function.  As usual, the pressure term can be eliminated by considering the projection of the system onto the space $\mathcal{H}$ which has been introduced in \eqref{DEFH}. We obtain in this way that \eqref{equation:euler_finitedimcontrol} can be rephrased as: 
\begin{align}
	\frac{\partial v}{\partial t}+N(v) =h+\sum_{j=1}^m u_j \theta_j && \text{in $(0,\infty)\times \mathbb{T}^2$}, \label{E:2.3} \\
	v(0,\cdot) = v_0 && \text{in $\mathbb{T}^2$}, \label{E:2.2_bis}
\end{align}
where the nonlinear operator $N$ has been defined in \eqref{nonlinear}.

\subsubsection*{Approximate tracking controllability}

In what follows, we fix $T>0$ and assume that $h\in L^1(J_T;H^{k+2}_\sigma)$. For the well-posedness of the system \eqref{E:2.3} and \eqref{E:2.2_bis} we refer to the papers \cites{Wol33,Gol66,McG67,Kato67,Kato-72} or to Chapter~17 in~\cite{T23}. In particular, when $k\geqslant 3$, it is known that for any $v_0 \in H^k_\sigma$ and~$u\in L^1(J_T;\mathbb{R}^m)$, there is a unique 
solution~$v\in C(J_T;H^k_\sigma)\cap W^{1,1}(J_T;H^{k-1}_\sigma)$. Given any $f\in H^k$ and $v\in C(J_T;H^k_\sigma)$, we  denote by $\varphi_v$ the unique solution of the problem \eqref{a_doua_advectie} and \eqref{CI_prima}.

\begin{theorem}\label{T.2.1}
	For $k\geqslant 2$, let $E\subset H^{k+4}_\sigma$ be a saturating subspace. Then, for any $\varepsilon>0$, $\psi\in  W^{1,1}(J_T;H^{k+2}_\sigma)\cap L^1(J_T;H^{k+3}_\sigma)$, and $R>0$, there is a control $u\in C^\infty(J_T;\mathbb{R}^m)$ such that
	$$
	\|v(T) - \psi(T) \|_{k+2}  + |||v-\psi|||_{T,k}+\|\varphi_v -  \varphi_\psi \|_{C(J_T;H^k)}<\varepsilon,
	$$
	uniformly with respect to $f\in H^{k+2}$ with $\|f\|_{k+2}\leqslant R$, where $v$ is the solution of~\eqref{E:2.3} with initial condition $v(0)=\psi(0)$,
	$\varphi_v$ is the solution of the problem \eqref{a_doua_advectie} and \eqref{CI_prima}, and $\varphi_\psi$ is the solution of the problem \eqref{a_doua_advectie} and \eqref{CI_prima} with $v=\psi$.  
\end{theorem}

In order to prove the above result, we need some new notation and several preliminaries.

Firstly, for every~$\ell=[l_1,l_2]^\top\in {\mathbb Z}^2_*$, we define the functions
$$
c_\ell(x) =  \ell^\bot \cos(\ell \cdot x ), \quad s_\ell(x) = \ell^\bot
\sin(\ell \cdot x ) \qquad (x\in \mathbb{T}^2),
$$
where  $\ell^\bot \coloneqq [-l_2, l_1]^\top$. We note that the family $\{c_\ell, s_\ell\}_{\ell\in {\mathbb Z}^2_*}$ is a complete orthogonal system in $H^k_\sigma$. For any finite subset ${\cal K}\subset {\mathbb Z}^2_*$, let    
\begin{equation}\label{E:2.4}
	E({\cal K}) \coloneqq \textup{span} \{ c_\ell,s_\ell: \ell\in {\cal K}\}.
\end{equation}
We  say that~${\cal K}$ is a {\it generator}  if any vector in ${\mathbb Z}^2$ can be expressed as a finite integer linear combination of elements from ${\cal K}$. 
The following result provides a characterization of saturating spaces of the form  \eqref{E:2.4}, see Section~4.1 in~Hairer and Mattingly~\cite{HM-2006} for a proof.  

\begin{lemma}\label{lem:HM}
	The space $E({\cal K})$ is saturating if and only if~${\cal K}$ is a generator of~${\mathbb Z}^2$  and contains two non-parallel vectors of different lengths.
\end{lemma}

A simple example of ${\cal K}$ satisfying the assumptions in the lemma above is given below:

\begin{example}\label{example:saturating}
	Let ${\cal K}=\{[1,0]^\top;[1,-1]^\top\}$. Then, a saturating four-dimensional space is given by
	\begin{equation*}
		E({\cal K}) = \textup{span} \{ \theta_1,\theta_2,\theta_3,\theta_4 \},
	\end{equation*}
	where $\theta_1,\dots,\theta_4$ are the functions introduced in \Cref{cor_first}.
\end{example}

\subsection{Proof of Theorem \ref{T.2.1}} \label{S:reduct}

In this subsection, we obtain \Cref{T.2.1} as a consequence of an auxiliary result (Theorem~\ref{T.reduct}) which states that any solution $v$ of~\eqref{E:2.3} and~\eqref{E:2.2_bis} with a control $u$ in $E_1=\mathcal{F}(E)$, together with the solution $\varphi_v$ of~\eqref{a_doua_advectie} and \eqref{CI_prima}, can be approximately tracked using the corresponding solutions with a control $u$ in $E$.

\subsubsection*{Setup and well-posedness}

Together with \eqref{E:2.3} and \eqref{E:2.2_bis} we consider the system
\begin{align}
	\frac{\partial w}{\partial t}+N(w+\zeta) =h+u && \text{in $[0,T]\times \mathbb{T}^2$}, \label{E:2.3b_bis} \\
	w(0,\cdot) = w_0 && \text{in $\mathbb{T}^2$} \label{E:2.3b_init_bis}
\end{align}
with two controls $\zeta=\zeta(t,x)$ and $u=u(t,x)$. The following well-posedness and stability result can be proved by applying the methods of the references \cites{Wol33,Gol66,McG67,Kato67,Kato-72,T23}; see also Theorem~2.1 in~\cite{nersisyan-2010} for a similar result in the context of the 3D Euler system.

\begin{proposition}\label{P:4.1}
	Let $k\geqslant 3$ be an integer. Then for any $T>0$, $w_0\in H^k_\sigma$, $\zeta\in L^2(J_T; H^{k+1}_\sigma)$, and $u\in L^1(J_T;H^k_\sigma)$, there is a unique solution $w\in C(J_T; H^k_\sigma)\cap 
	W^{1,1}(J_T;H^{k-1}_\sigma)$ to the problem \eqref{E:2.3b_bis} and \eqref{E:2.3b_init_bis}. Furthermore, the following properties~hold:
	\begin{itemize}
		\item[\rm (i)]  The operator 
		$ \mathcal{R}$  mapping a triple $(w_0,\zeta,u)$ to the solution $w$ is continuous
		from $H^k_\sigma\times L^2(J_T; H^{k+1}_\sigma)\times L^1(J_T;H^k_\sigma)$ to $C(J_T; H^k_\sigma)\cap 
		W^{1,1}(J_T;H^{k-1}_\sigma)$.  
		\item[\rm (ii)] The operator $ \mathcal{R}$ is locally Lipschitz with respect to appropriately chosen norms. More precisely,  for any $M>0$, there is a constant $C>0$ such that 
		\begin{align*}
			& \|{\cal R}(w_0,\zeta, u)-{\cal R}(v_0,\xi, g)\|_{C(J_T;H^{k-1}_\sigma)} \\
			& \quad \leqslant C\left( \|w_0-v_0\|_{H^{k-1}_\sigma}+\|\zeta-\xi\|_{L^2(J_T;H^k_\sigma)}+\|u-g\|_{L^1(J_T;H^{k-1}_\sigma)} \right),
		\end{align*}
		provided that
		\begin{align*}
			\|w_0\|_k&+\|v_0\|_k+\|\zeta\|_{L^2(J_T; H^{k+1}_\sigma)}+\|\xi\|_{L^2(J_T; H^{k+1}_\sigma)} \\&+\|u\|_{L^1(J_T; H^{k}_\sigma)} +\|g\|_{L^1(J_T; H^{k}_\sigma)}\leqslant M.
		\end{align*}
	\end{itemize}
\end{proposition}

For $T > 0$, let ${\cal R}_T$ denote the operator $(w_0,\zeta, u) \mapsto {\cal R}_T(w_0,\zeta, u)=w(T,\cdot)$,
where $w$ solves \eqref{E:2.3b_bis} and~\eqref{E:2.3b_init_bis} with initial state $w_0$ and forces $(\zeta, u)$. Moreover, we introduce the resolving operator of the problem \eqref{a_doua_advectie} and \eqref{CI_prima} by
\begin{align*}
	\Psi \colon H^{k+2}\times C(J_T;H^{k+2}_\sigma)&\to C(J_T;H^{k+2})\cap L^2(J_T;H^{k+3})\cap H^1(J_T;H^{k+1}) \\
	(f,v)&\mapsto \varphi_v,
\end{align*} 
which is well-defined thanks to Proposition~\ref{th_exist_unique}.

\subsubsection*{An auxiliary result}
The proof of \Cref{T.2.1} uses the following auxiliary result which states that any trajectory produced by a control $u$ in $E_1=\mathcal{F}(E)$ can be approximately tracked using a control $u$ in $E$. Its proof is given in \Cref{subsection:proofauxthm} below.

\begin{theorem}\label{T.reduct}  Assume that~$E$ is a finite-dimensional subspace of~$H_\sigma^{k+4}$ and let $\{ E_j \}_{j\in \mathbb{N}\cup\{0,\infty\}}$ be the sequence of subspaces defined in
	formula \eqref{spatii_stranii}. Then, for any $\varepsilon>0$, $R>0$, $w_0\in H^{k+2}_\sigma$, and $u_1\in L^1(J_T; E_1)$, there exists $u \in C^\infty(J_T;E)$ such that
	\begin{gather*}
		\|{\cal R}_T(w_0,0,u_1) - {\cal R}_T(w_0,0,u) \|_{k+2}+|||  {{\cal R}(w_0,0,u_1)} -  {{\cal R}(w_0,0,u)}|||_{T,k} \\
		\quad +\|\Psi(f,{\cal R}(w_0,0,u_1)) - \Psi(f,{\cal R}(w_0,0,u)) \|_{C(J_T;H^k)}<\varepsilon
	\end{gather*}
	uniformly with respect to $f \in H^{k+2}$ with $\|f\|_{k+2}\leqslant R$.
\end{theorem}

Taking \Cref{T.reduct} for granted, \Cref{T.2.1} is now obtained by developing the arguments of the papers \cites{AS-05,AS-2006,shirikyan-cmp2006, shirikyan-ihp-2007, nersisyan-2010, Ners-2015}.

\begin{proof}[Proof of Theorem \ref{T.2.1}]
	For any $\psi\in  W^{1,1}(J_T;H^{k+2}_\sigma)\cap L^1(J_T;H^{k+3}_\sigma)$, the control $u_*$ defined by 
	$$
	u_* \coloneqq \dot \psi+N(\psi)-h
	$$ belongs to $L^1(J_T;H^{k+2}_\sigma)$ and that $\psi={\cal R}(w_0,0,u_*)$, where $w_0=\psi(0)$. Since $E_\infty$ is dense in $H^{k+2}_\sigma$, we deduce that
	$$ \|\mathrm{P}_{E_N} u_{*} - u_{*}\|_{L^1(J_T;H^{k+2}_\sigma)} \rightarrow 0 \,\, \text{as}\,\,N \rightarrow \infty,$$
	where $\mathrm{P}_{E_N}$ is the orthogonal projection onto $E_N$ in $H_{\sigma}^{k+2}$.
	From Proposition~\ref{P:4.1} it follows that      
	$$
	\|  {{\cal R}(w_0,0,\mathrm{P}_{E_N} u_* )} - \psi \|_{C(J_T;H^{k+2}_\sigma)}  \rightarrow 0 \,\, \text{as}\,\,N \rightarrow \infty.
	$$
	This and Theorem~\ref{L:1} imply that 
	\begin{align*} 
		& \|{\cal R}_T(w_0,0,\mathrm{P}_{E_N} u_*)-\psi(T) \|_{k+2}  + |||{{\cal R}(w_0,0,\mathrm{P}_{E_N} u_*)} - \psi |||_{T,k}\nonumber \\
		& \quad +\|\Psi(f,{\cal R}(w_0,0,\mathrm{P}_{E_N} u_*))-  \Psi(f,\psi) \|_{C(J_T;H^k)}\to0\,\, \text{as}\,\,N \rightarrow \infty
	\end{align*}
	uniformly with respect to $f\in H^{k+2}$ with $\|f\|_{k+2}\leqslant R$. Therefore, for any $ \varepsilon>0$ and sufficiently large $N\geqslant 1$, 
	\begin{gather}\label{eq:target_N}
		\|{\cal R}_T(w_0,0,\mathrm{P}_{E_N} u_*) - \psi(T) \|_{k+2}  + |||  {{\cal R}(w_0,0,\mathrm{P}_{E_N} u_*)} - \psi |||_{T,k}\nonumber\\
		\quad +\|\Psi(f,{\cal R}(w_0,0,\mathrm{P}_{E_N} u_*))-  \Psi(f,\psi) \|_{C(J_T;H^k)}< \varepsilon
	\end{gather}
	for any $f\in H^{k+2}$ with $\|f\|_{k+2}\leqslant R$. Now, applying Theorem~\ref{T.reduct} with $\varepsilon$ replaced by $\varepsilon/N$, we find $u_{N-1}\in C^{\infty}(J_T;E_{N-1})$, where $\{ E_j \}_{j\in \mathbb{N}\cup \{ 0,\infty \}}$ is defined in~\eqref{spatii_stranii}, such that
	\begin{align}
		& \|{\cal R}_T(w_0,0,\mathrm{P}_{E_N} u_*)  - {\cal R}_T(w_0,0,u_{N-1}) \|_{k+2} \nonumber\\ 
		& \quad \quad+ |||  {{\cal R}(w_0,0,\mathrm{P}_{E_N} u_*)} - {{\cal R}(w_0,0,u_{N-1})} |||_{T,k}\nonumber\\ 
		& \quad\quad +\|\Psi(f,{\cal R}(w_0,0,\mathrm{P}_{E_N} u_*))-\Psi(f,{\cal R}(w_0,0,u_{N-1}))\|_{C(J_T;H^k)}< \varepsilon/N.\label{eq:N_Nminus1}
	\end{align}
	Similarly, there exists $u_j \in C^{\infty}(J_T;E_j)$ ($j=0,1,\dotsc,N-2$) such that
	\begin{gather}
		\|{\cal R}_T(w_0,0,u_{j+1}) - {\cal R}_T(w_0,0,u_{j}) \|_{k+2}  + |||  {{\cal R}(w_0,0,u_{j+1})} - {{\cal R}(w_0,0,u_{j})} |||_{T,k} \nonumber \\
		\quad +\|\Psi(f,{\cal R}(w_0,0,u_{j+1}))-\Psi(f,{\cal R}(w_0,0,u_{j}))\|_{C(J_T;H^k)}< \varepsilon/N. \label{eq:jplus1_j}
	\end{gather}
	Then combining~\eqref{eq:target_N}--\eqref{eq:jplus1_j}, we obtain
	\begin{gather*}
		\|{\cal R}_T(w_0,0,u_{0}) - \psi(T) \|_{k+2}  + |||  {{\cal R}(w_0,0,u_{0})} - \psi |||_{T,k}\nonumber\\
		\quad +\|\Psi(f,{\cal R}(w_0,0,u_0))-\Psi(f,\psi)\|_{C(J_T;H^k)}< 2\varepsilon.
	\end{gather*}
	This shows that the conclusion of Theorem~\ref{T.2.1} holds with the control
	$$
	u \coloneqq u_0 \in C^{\infty}(J_T;E) \cong C^{\infty}(J_T;\mathbb{R}^m)
	$$
	and the corresponding solution $v={\cal R}(w_0,0,u_{0})$.
\end{proof}

\subsection{Proof of Theorem \ref{T.reduct}}\label{subsection:proofauxthm}

Here we establish Theorem \ref{T.reduct}, which was an essential ingredient of the proof of Theorem \ref{T.2.1}. More precisely, we derive Theorem \ref{T.reduct} from the following proposition, which will be proven in the next subsection.

\begin{proposition}\label{P.2}
	For any piecewise constant function  $u_1 \colon J_T\to E_1$, there is a sequence $\{(u_n,\zeta_ n)\}_{n\in \mathbb{N}} \subset   C^\infty(J_T;E\times
	E)$ such that 
	\begin{equation} 
		\sup_{n\in \mathbb{N}}\left(\| \zeta_n\|_{C(J_T; H^{k+2}_\sigma)}+\|u_n\|_{L^2(J_T; H^{k+2}_\sigma)}\right) < \infty\label{E:3.1}
	\end{equation}
	and
	\begin{equation}\label{E:3.2}
		\| {\cal R}(w_0,0,u_1)- {\cal R}(w_0, \zeta_n,u_n)\|_{C(J_T; H^{k+2}_\sigma)}+  |||\zeta_n|||_{T,k}  \to 0    \,\,\text{as $n\to\infty$}  
	\end{equation}for any $w_0\in H^{k+2}_\sigma$.
\end{proposition}

\begin{proof}[Proof of Theorem \ref{T.reduct}] 
	Thanks to Proposition \ref{P:4.1}, it suffices to prove that the conclusion of Theorem \ref{T.reduct}
	holds for $u_1$ piecewise constant. In this case we can apply Proposition \ref{P.2} to obtain the existence of a 
	sequence $\{(u_n,\zeta_ n)\}_{n\in \mathbb{N}} \subset C^\infty(J_T;E\times
	E)$  satisfying \eqref{E:3.1} and \eqref{E:3.2}. 
	
	Using Proposition~\ref{P:4.1}, we can find $\hat \zeta_n \in C^\infty(J_T;E)$ such that $\hat\zeta_n(0)=\hat\zeta_n(T)=0$ and 
	\begin{gather}
		\|\zeta_n-\hat\zeta_n\|_{L^2(J_T;H^{k+3}_\sigma)} \rightarrow 0 \,\, \text {as} \,\,
		n\rightarrow \infty,\label{E:3.3}\\   \sup_{n\in \mathbb{N}}\| \hat\zeta_n\|_{C(J_T;H^{k+2}_\sigma)} < \infty,\label{E:3.4}\\
		\| {\cal R}(w_0, \zeta_n,u_n)-{\cal R}(w_0, \hat \zeta_n,u_n)\|_{C(J_T;H^{k+2}_\sigma)}  \to 0 \, \, \text {as $n\to\infty$.} \label{E:3.6}
	\end{gather}
	Then \eqref{E:3.2} and \eqref{E:3.3} imply that
	\begin{align}\label{E:3.5}
		||| \hat \zeta_ n |||_{ T,  k}& \leqslant   ||| \hat \zeta_ n-\zeta_ n |||_{ T,  k}+|||   \zeta_ n |||_{ T,  k}  \nonumber \\
		&  \leqslant \int_0^T \|\hat \zeta_ n(s)-\zeta_ n(s)\|_k \, \diff{s}+ |||   \zeta_ n |||_{T,k}   \to 0  \, \, \textup{as $n\to \infty$.}
	\end{align}
	Note that    
	\begin{align}\label{E:3.7}
		{\cal R}_t(w_0, \hat \zeta_n,u_n)&={\cal R}_t(w_0,0,\hat u_n)-\hat \zeta_n(t)  \quad (t\in J_T),\\
		{\cal R}_T(w_0, \hat \zeta_n,u_n)&={\cal R}_T(w_0,0, \hat u_n),\label{E:3.8}
	\end{align} where
	$\hat u_n \coloneqq u_n+\partial_t 
	{\hat \zeta}_n$.   From \eqref{E:3.2}, \eqref{E:3.6}, and \eqref{E:3.8} it follows  that 
	$$
	\| {\cal R}_T(w_0,0,u_1)-{\cal R}_T(w_0,0,\hat u_n)\|_{k+2}  \to 0 \, \, \text {as $n\to\infty$.
	}   $$  
	Using \eqref{E:3.2}, \eqref{E:3.6}--\eqref{E:3.7}, we obtain
	\begin{align}\label{E:rnl}
		||| {\cal R}(w_0, 0,u_1)-{\cal R}(w_0,0,\hat u_n)|||_{T,k}
		& \leqslant  T \|{\cal R}(w_0,0,u_1)-{\cal R}(w_0, \zeta_n,u_n)\|_{C(J_T; H^k_\sigma)}\nonumber\\
		& \quad +T\| {\cal R}(w_0, \zeta_n, u_n)-{\cal R}(w_0, \hat\zeta_n,u_n)\|_{C(J_T; H^k_\sigma)}\nonumber\\
		& \quad +|||   {{\cal R}(w_0, \hat \zeta_n,u_n)} -  {{\cal R}(w_0,0,\hat u_n)}|||_{ T,  k}\nonumber\\&
		\to 0 \, \, \text{as $n\to \infty$}.
	\end{align} 
	From \eqref{E:3.2}, \eqref{E:3.4}, \eqref{E:3.6}, and \eqref{E:3.7} it follows that $\{{\cal R}(w_0,0,\hat u_n)\}_{n\geqslant 1}$ is bounded in $C(J_T;H^{k+2}_\sigma)$. 
	Combining this with \eqref{E:rnl} and Theorem \ref{L:1}, we~get that 
	$$   \|\Psi(f,{\cal R}(w_0, 0,u_1)) - \Psi(f,{\cal R}(w_0,0,\hat u_n)) \|_{C(J_T;H^k)}\to 0 \, \, \text{as $n\to \infty$}
	$$uniformly with respect to $f \in H^{k+2}$ with $\|f\|_{k+2}\leqslant R$.
	This ends the proof of Theorem~\ref{T.reduct}. 
\end{proof}

\subsection{Proof of Proposition \ref{P.2}}
In this subsection, we obtain \Cref{P.2}, which is the missing link in the proof of \Cref{T.reduct}.

\begin{proof}[Proof of Proposition \ref{P.2}] {\it Step~1}.~Without loss of generality, we can assume that~$u_1\in E_1$ is constant. Indeed, the general case can be obtained by successively applying the result to each interval of constancy, approximating the resulting controls with smooth ones, and using Proposition \ref{P:4.1}. Furthermore, we will first assume that the initial data is more regular, that is,~$w_0\in H^{k+4}_\sigma$.
	
	By the definition of $E_1={\cal F}(E)$  in \eqref{spatii_stranii}, any $u_1\in E_1$ can be represented in the form
	$$
	u_1=u-\sum_{i=1}^p N(\xi^i)
	$$ for some integer $p\geqslant 1$ and vectors $\xi^1,\ldots,\xi^p, u \in E$. Choosing $m = 2p$ and setting
	$$
	\zeta^i \coloneqq \sqrt{\frac{m}{2}}\xi^i, \quad   \zeta^{i+p} \coloneqq -\sqrt{\frac{m}{2}}\xi^i, \quad  i = 1,
	\ldots,p,
	$$ it is easy to see that  
	$$
	N(w)-u_1=\frac1m \sum_{j=1}^m N(w+\zeta^j)-u \qquad
	(w \in H^{k+4}_\sigma).
	$$
	Then  $w_1 \coloneqq {\cal R}(w_0,0,u_1)\in C(J_T;H^{k+4}_\sigma)$   satisfies the following  equation
	\begin{equation}\label{E:3.9}
		\dot{w}_1+\frac1m\sum_{j=1}^m     N(w_1+\zeta^j)=h(t)+ u.
	\end{equation}
	Let $\zeta(t)$ be a $1$-periodic function such that
	$\zeta(s)=\zeta^j$ for  $s\in \left[ {(j-1)}/{m},  j/m\right)$ and  $j=1, \ldots, m$, and let  
	$\zeta_n(t)=\zeta(nt/T)$. 
	The equation \eqref{E:3.9} is
	equivalent to  
	\begin{equation}
		\dot w_1   + N(w_1 + \zeta_n)  = h(t) + u 
		+ g_n(t),\nonumber
	\end{equation}
	where
	$$
	g_n(t) \coloneqq N(w_1+\zeta_n)-\frac1m\sum_{j=1}^m N(w_1+\zeta^j) .
	$$
	For any $g\in L^2(J_T; \mathcal{H})$, let us set $K g(t) \coloneqq \smallint_0^t  g(s) \, \diff{s}$ and note that the difference $v_n \coloneqq w_1-Kg_n$ solves the problem
	\begin{equation}\label{eq_vn}
		\dot v_n  +N (v_n + \zeta_n + K g_n)   = h(t)
		+ u , \quad  v_n(0)=w_0.\nonumber
	\end{equation}
	Hence $v_n={\cal R}(w_0,\zeta_n+Kg_n,u)$. Next, the definition of $\zeta_n$ implies that the set $\{\zeta_n(t)\}_{t \in J_T}$ is contained in a finite
	subset of $H_\sigma^{k+4}$ not depending on $n$; therefore, 
	\begin{equation}\label{E:3.10}
		\sup_{n\geqslant 1} \| \zeta_n\|_{L^\infty(J_T;H^{k+4}_\sigma)}  < \infty.
	\end{equation}
	In Step 2, we will show that 
	\begin{equation}\label{E:3.11}
		\|K g_n \|_{C(J_T;H^{k+3}_\sigma)}\rightarrow 0 \, \, \text{as $n\to\infty$}.
	\end{equation}
	This limit implies that $\sup_{n\in \mathbb{N}} \|v_n \|_{C(J_T;H^{k+3}_\sigma)}<\infty$. Hence, by \eqref{E:3.10},~\eqref{E:3.11}, and Proposition~\ref{P:4.1},
	we have      
	$$
	\|{\cal R}(w_0,\zeta_n,u)-v_n \|_{C(J_T;H^{k+2}_\sigma)}\rightarrow 0 \, \, \text{as} \,\,
	n\rightarrow \infty.
	$$   
	On the other hand, by \eqref{E:3.11} and the fact that $v_n=w_1-Kg_n$, one has the convergence~$\|v_n-w_1\|_{C(J_T;H^{k+3}_\sigma)}\rightarrow 0$ as $n\rightarrow \infty$. As a result,
	\begin{equation}\label{eq:Step1_conclusion}
		\|{\cal R}(w_0,\zeta_n,u)-w_1 \|_{C(J_T;H^{k+2}_\sigma)}\rightarrow 0 \, \, \text{as} \,\,
		n\rightarrow \infty.
	\end{equation}
	
	\smallskip
	{\it Step~2}. In this step, we prove the limit \eqref{E:3.11}.~By an approximation argument, it suffices to consider the case where $w_1 \colon J_T\to H_\sigma^{k+4}$ is piecewise constant.

	Let us first note that the family $\{K g_n\}_{n\in \mathbb{N}}$ is relatively compact in the space $C(J_T;H_\sigma^{k+3})$.
	Indeed, as $w_1$ is piecewise constant, the set $\{g_n(t)\}_{t \in J_T}$ is contained in a finite
	subset of $H_\sigma^{k+3}$ not depending on $n$. Therefore, there is a
	compact set $G \subset H_\sigma^{k+3}$ such that
	$K g_n(t) \in G$  for any $t \in J_T$  and   $n \geqslant 1$.
	As
	\begin{equation}
		\sup_{n\geqslant 1}\| g_n\|_{C(J_T; H^{k+3}_\sigma)} < \infty,\nonumber
	\end{equation}
	the family $\{K g_n\}_{n\in \mathbb{N}}$ is uniformly equicontinuous on $J_T$.
	The Arzel\`{a}--Ascoli theorem implies that $\{K g_n\}_{n\in \mathbb{N}}$ is
	relatively compact in $C(J_T; H_\sigma^{k+3})$. Thus, the limit \eqref{E:3.11} will be
	established if we show that
	\begin{equation}\label{E:3.12}
		\|K g_n(t)\|_{ H_\sigma^{k+3}} \rightarrow 0 \, \, \text{as} \,\,
		n\rightarrow \infty \, \, \text{for
			any $t \in J_T$}.
	\end{equation}
	We first prove this convergence in the case when $w_1$ is constant:
	$w_1(t)=b\in
	H_\sigma^{k+4}$, $t\in J_T$. Let $t=t_l+\tau$, where
	$t_l=lT/n$, $l\in {\mathbb N}$, and $\tau\in [0,T/n)$. From
	the definition of $g_n$ and $\zeta_n$, we have
	$$
	\int_0^{\frac{lT}{n}}g_n(s) \, \diff{s}= \int_0^{\frac{lT}{n}} N(b +
	\zeta_n(s)) \, \diff{s}-\frac{lT}{mn}\sum_{j=1}^m  N
	(b+\zeta^j) =0,
	$$
	hence
	\begin{equation}
		K g_n(t)= \int_0^\tau
		N(b + \zeta_n(s)) \, \diff{s}-\frac{\tau}{m} \sum_{j=1}^m  N (b+\zeta^j).\nonumber
	\end{equation}
	Since  $\tau \rightarrow 0 $ as $n \rightarrow \infty $, we get \eqref{E:3.12}. In the same way, we can show 
	\eqref{E:3.12} for any piecewise constant $w_1$.

	\smallskip
	{\it Step~3}. 
	Let us   show that 
	\begin{equation}\label{E:3.13}
		||| \zeta_ n |||_{T,k+4} \to 0 \, \, \text{ as $n\to\infty$.}
	\end{equation} The argument is similar to the one in the previous step. It is enough to check that  
	\begin{enumerate}
		\item[(i)] $\{K\zeta_ n\}_{n\in \mathbb{N}}$ is relatively compact in $C(J_T; H^{k+4}_\sigma)$;
		\item[(ii)] $\|K\zeta_n(t)\|_{H^{k+4}_\sigma}\to 0$ as $n\to\infty$ for any $t\in J_T$.  
	\end{enumerate} 
	To prove (i), we use the Arzel\`{a}--Ascoli theorem.
	The set $\{\zeta_n(t)\}_{t \in J_T}$ is contained in a finite
	subset of~$H_\sigma^{k+4}$ not depending on $n$. This implies that  there is a
	compact set $F \subset H_\sigma^{k+4}$ such that
	$K \zeta_n(t) \in F$  for any $t \in J_T$  and $n \geqslant 1$.
	From~\eqref{E:3.10} it follows that 
	the sequence $ \{K\zeta_n\}_{n\in \mathbb{N}} $ is uniformly equicontinuous on $J_T$.~Thus, by the Arzel\`{a}--Ascoli theorem, $\{K\zeta_n\}_{n\in \mathbb{N}}$ is
	relatively compact in~$C(J_T; H^{k+4}_\sigma)$. 
	
	To prove (ii), we fix $t=t_l+\tau$, where
	$t_l= lT/n$, $l\in {\mathbb N}$, and $\tau\in [0,T/n)$. In view of the construction of 
	$\zeta_n$, we have that $K\zeta _n(lT/n)=0$. Combining this with~\eqref{E:3.10} and the fact that $\tau \rightarrow 0 $ as $n \rightarrow \infty$, we get (ii). This completes the proof of~\eqref{E:3.13}.

	Note that $\zeta _n$ is $E$-valued. Taking  an arbitrary sequence $\{\hat \zeta_n\}_{n\in \mathbb{N}}\subset C^\infty (J_T; E)$ such that 
	$$
	\|\zeta_n-\hat \zeta_n\|_{L^\infty(J_T;H^{k+4}_\sigma)}\to 0 \, \, \text{as $n\to \infty$},
	$$  
	we infer from~\eqref{eq:Step1_conclusion} that the conclusions of Proposition \ref{P.2} hold  for   $\{(u, \hat\zeta_n)\}_{n\in \mathbb{N}}\subset C^\infty(J_T; E\times E)$ in the case  $w_0\in H^{k+4}_\sigma$. Finally, an approximation argument, the inequality \eqref{E:3.10}, and Proposition~\ref{P:4.1} imply that the conclusions of Proposition~\ref{P.2} still hold in the case~$w_0\in H^{k+2}_\sigma$.
\end{proof}

\section{Exact controllability}\label{sec_exact_Euler}
Let $T > 0$ and fix a nonempty open set $\omegaup \subset \mathbb{T}^2$ such that $\mathbb{T}^2\setminus\omegaup$ is simply-connected. We consider the motion of a fluid with velocity $v\colon[0,T]\times\mathbb{T}^2\longrightarrow \mathbb{R}^2$ and pressure $p\colon[0,T]\times\mathbb{T}^2\longrightarrow \mathbb{R}$, satisfying the forced incompressible Euler system 
\begin{equation}\label{equation:controlled_euler}
	\begin{gathered}
		\partial_t v + (v \cdot \nabla) v + \nabla p = h + \mathbb{I}_{\omegaup} u, \quad
		\operatorname{div} v = 0, \quad
		v(0, \cdot) = v_0,
	\end{gathered}
\end{equation}
where $h$ denotes a known body force and $u$ is an interior control physically localized in~$\omegaup$.

The global exact boundary controllability of the two-dimensional incompressible Euler system has been shown for smoothly bounded domains by Coron~\cite{Coron1996Euler} using his \textit{return method} (see also Glass~\cite{Glass2000} for the three-dimensional case). Therefore, the global exact controllability of \eqref{equation:controlled_euler} is in principle known, but not explicitly documented in the literature for periodic boundary conditions and with given body force $h$. For the convenience of the reader, and still by using classical ideas from \cite{Coron1996Euler}, we give here a short proof under the simplifying assumption that $\mathbb{T}^2 \setminus \omegaup$ is simply-connected. When $\omegaup\subset \mathbb{T}^2$ is arbitrary, our short proof will not work, but the original argument from \cite{Coron1996Euler} could be directly adapted for the case $h = 0$. See also \cite{Fernandez-CaraSantosSouza2016} by Fern\'{a}ndez-Cara et al.~for related arguments for the inviscid Boussinesq system, the work \cite{rissel2024exactcontrollabilityidealmhd} by Rissel on global exact interior controllability of ideal MHD in bounded 2D domains, and \cite{CoronFursikov1996} by Fursikov and Coron for the global exact controllability to trajectories of the 2D incompressible Navier--Stokes system on manifolds without boundary.

The following local null controllability result will be proved in \Cref{subsection:proof_local} by using Coron's return method. The global exact controllability of \eqref{equation:controlled_euler} will then follow as a corollary. Throughout, for $s > 0$ we denote
\begin{equation}\label{eq:Vs}
	V_s \coloneq C([0,s]; H^1(\mathbb{T}^2; \mathbb{R}^2)) \cap L^{\infty}([0,s]; H^3(\mathbb{T}^2; \mathbb{R}^2)).
\end{equation}

\begin{theorem}\label{theorem:EulerExactLocal}
	There exists~$\delta_0 \in (0,1)$ such that, given any divergence-free initial state $v_0 \in H^3(\mathbb{T}^2; \mathbb{R}^2)$ and force
	$h \in L^2([0,1]; H^5(\mathbb{T}^2;\mathbb{R}^2))$ meeting the smallness constraint
	\begin{equation}\label{equation:smallnessconstraint}
		\|v_0\|_3 + \|h\|_{L^2([0,1]; H^5(\mathbb{T}^2;\mathbb{R}^2))} < \delta_0,
	\end{equation}
	there is a control $u \in L^2([0,1]; H^2(\mathbb{T}^2; \mathbb{R}^2))$ for which the respective solution $v \in V_1$ to the incompressible Euler system \eqref{equation:controlled_euler} with $T = 1$ satisfies $v(1, \cdot) = 0$.
\end{theorem}

In view of known hydrodynamic scaling properties of the incompressible Euler system, the global exact controllability of \eqref{equation:controlled_euler} is now a direct consequence of \Cref{theorem:EulerExactLocal}.

\begin{theorem}\label{theorem:EulerExactGlobal}
	The system \eqref{equation:controlled_euler} is globally exactly controllable. That is, given any time $T > 0$, divergence-free states $v_0, v_T \in H^3(\mathbb{T}^2; \mathbb{R}^2)$, and a force $h \in L^2([0,T]; H^5(\mathbb{T}^2;\mathbb{R}^2))$, there is a control $u \in L^2([0,T]; H^2(\mathbb{T}^2; \mathbb{R}^2))$ for which the solution $v \in V_T$ to \eqref{equation:controlled_euler} obeys $v(T,\cdot) = v_T$.
\end{theorem}

\begin{proof}
	Since the incompressible Euler system on the $2$D flat torus is well-posed (see \Cref{extremely_difficult}), we can denote by $(\mathscr{v},\mathscr{p})$ the unique solution to
	\begin{equation}\label{equation:controlled_euler_uncontrolled}
		\begin{gathered}
			\partial_t \mathscr{v} + (\mathscr{v} \cdot \nabla) \mathscr{v} + \nabla \mathscr{p} = h, \quad
			\operatorname{div} \mathscr{v} = 0, \quad
			\mathscr{v}(0, \cdot) = v_0.
		\end{gathered}
	\end{equation}
	Now, fix any number $0 < \varepsilon < T/3$ so small that
	\begin{equation*}
		\sup_{t\in[0,T]}\varepsilon \|\mathscr{v}(t)\|_3 + \varepsilon\|v_T\|_3 + \varepsilon^{3/2}\| h\|_{L^2([0,T]; H^5(\mathbb{T}^2;\mathbb{R}^2))} < \delta_0,
	\end{equation*}
	where $\delta_0 > 0$ is a small number fixed in \Cref{theorem:EulerExactLocal}. By applying \Cref{theorem:EulerExactLocal} twice, using as initial data and external forces the respective pairs
	\begin{gather*}
		\left(\varepsilon \mathscr{v}(T-2\varepsilon), t \mapsto (\varepsilon^2 h(T-2\varepsilon + \varepsilon t )\right), \quad
		\left(\varepsilon v_T, t \mapsto \varepsilon^2 h(T-\varepsilon t)\right),
	\end{gather*}
	there are corresponding controls~$\widetilde{u}$ and~$\widehat{u}$ for which the associated solutions~$\widetilde{v}$ and~$\widehat{v}$ to~\eqref{equation:controlled_euler} satisfy $\widetilde{v}(1,\cdot) = \widehat{v}(1,\cdot) = 0$ in $\mathbb{T}^2$. Finally, we obtain the desired control~$u$ and solution~$v$ by means of
	\begin{gather*}
		(v, u)(t) \coloneq \begin{cases}
			(\mathscr{v}, 0)(t) & \mbox{ if } t \in [0, T-2\varepsilon],\\
			(\varepsilon^{-1}\widetilde{v}, \varepsilon^{-2}\widetilde{u})(\varepsilon^{-1}(t-T+2\varepsilon)) & \mbox{ if } t\in [T-2\varepsilon,T-\varepsilon],\\
			(\varepsilon^{-1}\widehat{v}, \varepsilon^{-2}\widehat{u})(\varepsilon^{-1}(T-t)) & \mbox{ if } t\in [T-\varepsilon, T].
		\end{cases}
	\end{gather*}
\end{proof}
Let us remark that the regularity stated for controlled solutions in Theorems~\ref{theorem:EulerExactGlobal} and~\ref{theorem:EulerExactLocal} could be immediately improved by using the equation \eqref{equation:controlled_euler}.
\subsection{Auxiliaries}
Let $T > 0$ and $v \in C([0,T]\times\mathbb{T}^2;\mathbb{R}^2)$ be Lipschitz continuous in the space variables with time-independent Lipschitz constant. We denote by $\Phi^v\colon \mathbb{T}^2\times[0,T]^2 \longrightarrow \mathbb{T}^2$ the flow of~$v$, assigning to each $(x,s,t) \in \mathbb{T}^2\times[0,T]^2$ the solution $\Phi^v(x;s,\cdot)$ to
\begin{equation}\label{equation:flow}
	\begin{gathered}
		\frac{\diff{{}}}{\diff{t}} \Phi^v (x;s,t) = v(t, \Phi^v(x;s,t)), \quad
		\Phi^v (x;s,s) = x
	\end{gathered}
\end{equation}
with initial condition $x$ imposed at the time $s$.

The following lemma, which is essentially \cite{NersesyanRissel2024}*{Theorem 3.2}, provides a smooth and constant-in-space vector field inducing a flow that flushes information through the control region. This vector field will serve as a reference trajectory in the sense of the return method. See also Fursikov and Imanuvilov~\cite{FursikovImanuvilov1999}*{Lemma 5.1} regarding return method trajectories on $2$D and $3$D flat tori that work for general control regions but are not constant with respect to the space variables.

\begin{lemma}\label{lemma:reference_trajectory}
	Let $T > 0$ and $\widetilde{\omegaup} \subset \mathbb{T}^2$ with nonempty interior. There exists a function $\overline{y} \in C^{\infty}_0((0,T);\mathbb{R}^2)$ such that
	the flow $\Phi^{\overline{y}}$ obtained by solving \eqref{equation:flow} with $v=\overline y$ satisfies $\Phi^{\overline{y}} (x;0,T) = x$ for all $x \in \mathbb{T}^2$ and
	\begin{equation}\label{equation:flushingproperty}
		\forall  x \in \mathbb{T}^2,  \, \exists t_x \in (0,T) \colon \Phi^{\overline{y}} (x;0,t_x) \in \widetilde{\omegaup}.
	\end{equation}
\end{lemma}

\begin{proof}
	We fix an open ball $O \subset \widetilde{\omegaup}$ with center $x^0$ and cover $\mathbb{T}^2$ by a finite number $M \in \mathbb{N}$ of translations $\{ O_i \}_{i \in \{1,\dots,M\}}$ of $O$. For each $i \in \{1,\dots,M\}$, we denote the center of $O_i$ by $x^i$ and make any choice of functions $\chi_i \in C^{\infty}_0((0,T/M);\mathbb{R}^2)$ such that
	\begin{gather*}
		\int_0^{\frac{T}{2M}} \chi_i(s) \, \diff{s} = x^0 - x^i, \quad \int_{\frac{T}{2M}}^{\frac{T}{M}} \chi_i(s) \, \diff{s} = x^i - x^0
	\end{gather*}
	for $i \in \{1, \dots, M\}$. Then, for $t \in [(i-1)T/M, iT/M)$, with $i \in \{1, \dots, M\}$, we define
	\[
	\overline{y}(t) \coloneq \chi_i(t-(i-1)T/M).
	\]
	
	To verify \eqref{equation:flushingproperty}, let $x \in \mathbb{T}^2$ and $i \in \{1, \dots, M\}$ such that $x \in O_i$. Then, for the choice $t_x \coloneq T(2i-1)/2M$, integrating \eqref{equation:flow} yields
	\[
	\Phi^{\overline{y}}(x;0, t_x) = x + \int_{(i-1)T/M}^{T(2i-1)/2M} \chi_i(s-(i-1)T/M) \, \diff{s} = x^0 + x - x^i \in \widetilde{\omega}.
	\]
\end{proof}

The next lemma is concerned with solving the problem $\nabla\wedge F = f$ under conditions on the support of $f$ and $F$. The proof follows along the lines of Coron et al.~\cite{CoronMarbachSueur2020}*{Appendix A.2}.
\begin{lemma}\label{lemma:integrating_control}
	Let $S = (S_1, S_2)^2 \subset \mathbb{T}^2$,~$k \in \mathbb{N}$, and~$f \in H^k(\mathbb{T}^2)$ a zero average function with $\operatorname{supp}(f) \subset S$.
	There exists $F \in H^k(\mathbb{T}^2; \mathbb{R}^2)$ such that $\nabla\wedge F = f$ in $\mathbb{T}^2$ and $\operatorname{supp}(F) \subset S$.
\end{lemma}
\begin{proof}
	As the case $S = \mathbb{T}^2$ is trivial, we can assume that
	$\operatorname{supp}(f) \subset(L_1, L_2)^2$ for $0< S_1<L_1 < L_2<S_2 < 2\pi$. Further, let $\chi \in C^{\infty}(\mathbb{T})$ with $\operatorname{supp}(\chi) \subset (S_1, S_2)$ and $\chi = 1$ on a neighborhood of $[L_1, L_2]$, and we fix $\sigma \in C^{\infty}([0, 2\pi))$ such that $\sigma(s) = 0$ if $s \in [0, L_1]$ and $\sigma(s) = 1$ if $s \in [L_2,2\pi)$. Even though~$\sigma$ has jumps when periodically extended to $\mathbb{T}$, the product $\chi \sigma'$ can be viewed as a smooth function supported near $[L_1, L_2]$ modulo $2\pi$.
	Finally, the desired function $F = [F_1, F_2]^{\top} \in H^k(\mathbb{T}^2; \mathbb{R}^2)$ is obtained by defining its components for each $x \in \mathbb{T}^2$ by means of
	\begin{equation}\label{equation:F}
		\begin{gathered}
			F_1(x_1,x_2) \coloneq  -\chi(x_1)\sigma'(x_1)\int_{L_1}^{x_2} \int_{L_1}^{L_2} f(x_1, s) \, \diff{x}_1 \, \diff{s},\\
			F_2(x_1, x_2) \coloneq \chi(x_1)\left[- \sigma(x_1)\int_{L_1}^{L_2} f(x_1, x_2) \, \diff{x}_1 + \int_0^{x_1} f(s, x_2) \, \diff{s}\right].
		\end{gathered}
	\end{equation}
\end{proof}

\subsection{{Local exact controllability: proof of \Cref{theorem:EulerExactLocal}}}\label{subsection:proof_local}
We fix a return method trajectory $\overline{y} \in C^{\infty}_0((0,1); \mathbb{R}^2)$ by applying \Cref{lemma:reference_trajectory} for an open square $\widetilde{\omegaup} \subset \omegaup$ and $T = 1$. Then, we fix a suitable $M \in \mathbb{N}$ and denote by $O_1, \dots, O_M \subset \mathbb{T}^2$ an open covering of~$\mathbb{T}^2$ such that for each $i \in \{1,\dots,M\}$ there are $0 < t_i^a < t_i^b < 1$ with
\begin{equation}\label{equation:flushing}
	\Phi^{\overline{y}}(O_i; 0, [t_i^a, t_i^b]) \subset \widetilde{\omegaup},
\end{equation}
where we denote $\Phi^{\overline{y}}(A; 0, J) \coloneq \big\{\Phi^{\overline{y}}(x; 0, t) \, \, \big| \, \, x \in A, \, t \in J \big\}$
for $A \subset \mathbb{T}^2$ and $J \subset [0,1]$. Further, we introduce a partition of unity $\{ \mu_i \}_{i\in\{1,\dots,M\}} \subset C^{\infty}(\mathbb{T}^2)$ subordinate to~$\{ O_i \}_{i\in\{1,\dots,M\}}$. That is,
\[
\operatorname{supp}(\mu_i) \subset O_i, \quad \sum_{i=1}^M \mu_i = 1
\]
for $i\in\{1,\dots,M\}$.

The desired controlled solution to \eqref{equation:controlled_euler} with $T = 1$ will be obtained as the fixed point of a certain map~$\mathscr{F}$ defined on
\[
X_{\delta} \coloneq \left\{ v \in V_1 \, \, \Bigg| \, \, \sup_{t \in [0,1]}\|v(t,\cdot) - \overline{y}(t)\|_3 \leq \delta, \,\, \operatorname{div}v = 0\right\},
\]
where $V_1$ is defined in~\eqref{eq:Vs} and $\delta > 0$ is chosen so small that
\begin{equation}\label{equation:geometric_property}
	\Phi^{\widetilde{v}}(O_i; 0, [t_i^a, t_i^b]) \subset \widetilde{\omegaup}
\end{equation}
for each $\widetilde{v} \in X_{\delta}$ and $i \in \{1,\dots, M\}$.

\begin{remark}
	It is possible to choose~$\delta > 0$ such that \eqref{equation:geometric_property} holds for all $\widetilde{v} \in X_{\delta}$. Indeed, for $x \in \mathbb{T}^2$ and $(s,t)\in[0,1]^2$, one has
	\begin{align*}
		\sup_{x \in \mathbb{T}^2, (s,t) \in [0,1]^2} |(\Phi^{\overline{y}}-\Phi^{\widetilde{v}})(x;s,t)|
		\leqslant \int_{s}^t | \overline{y}(r) - \widetilde{v}(r, \Phi^{\widetilde{v}}(x;s,r)) | \, \diff{r} \leq \delta.
	\end{align*}
	Therefore, due to \eqref{equation:flushing}, one obtains \eqref{equation:geometric_property} whenever the number $\delta$ in the definition of $X_{\delta}$ is sufficiently small.
\end{remark}

\subsubsection{The fixed point map $\mathscr{F}$}\label{subsubsection:construction_F}
Let $\widetilde{v} \in X_{\delta}$ be arbitrary. In what follows, we define based on $\widetilde{v}$ a new element $v \in X_{\delta}$, and then assign~$\mathscr{F}(\widetilde{v}) \coloneq v$.

First, we denote by~$\{ w_i \}_{i \in \{1,\dots, M\}}$ the family of solutions to the respective linear transport problems 
\begin{equation}\label{equation:aux_vorticity}
	\partial_t w_i + (\widetilde{v} \cdot \nabla) w_i = 0, \quad w_i(0, \cdot) = \nabla \wedge (\mu_i v_0),
\end{equation}
where for each $i \in \{1,\dots, M\}$ the initial state of $w_i$ is localized in $O_i$. Moreover, by the method of characteristics, it holds
\begin{equation}\label{equation:mcw}
	w_i(t,x) = (\nabla \wedge (\mu_i v_0)) (\Phi^{\widetilde{v}}(x; t, 0)) 
\end{equation}
for $i \in \{1,\dots, M\}$. This implies the regularity
\[
w_1, \dots, w_M \in C([0,1]; L^2(\mathbb{T}^2))\cap L^{\infty}([0,1]; H^2(\mathbb{T}^2)).
\]

Below, in \eqref{equation:auxilliary_vorticity_combined}, we will employ the auxiliary functions $\{ w_i \}_{i \in \{1,\dots, M\}}$ as building blocks for a vorticity control.To that end, we note that for any $i\in\{1,\dots,M\}$ and $t \in [0,1]$, one has
\begin{equation}\label{equation:avg1}
	\int_{\mathbb{T}^2} w_i(t,x) \, \diff{x} = 0.
\end{equation}
Indeed, due to $\int_{\mathbb{T}^2} \nabla \wedge (\mu_i v_0)(x) \, \diff{x} = 0$, this follows by integrating the equation \eqref{equation:aux_vorticity} over $\mathbb{T}^2$ and applying integration by parts.

To account also for the prescribed body force~$h$ when defining the vorticity control in \eqref{equation:auxilliary_vorticity_combined} below, we introduce another auxiliary function~$\widetilde{w}_h$ as the solution to the transport problem
\begin{gather}\label{equation:aux_vorticity_h}
	\partial_t \widetilde{w}_h + (\widetilde{v} \cdot \nabla) \widetilde{w}_h = \nabla \wedge h + \widetilde{\xi}_h, \quad \widetilde{w}_h(0,\cdot) = 0,
\end{gather}
where
\begin{gather*}
	\widetilde{\xi}_h(t,x) \coloneq -\sum_{i=1}^M \frac{\mathbb{I}_{[t_i^a,t_i^b]}(t)}{t_i^b - t_i^a}\mu_i(\Phi^{\widetilde{v}}(x;t,1)) (\nabla \wedge h)\left(z_i(t), \Phi^{\widetilde{v}}\left(x; t, z_i(t)\right)\right),
\end{gather*}
with $z_i(t) \coloneq (t - t_i^a)/(t_i^b - t_i^a)$ for $i \in \{1,\dots, M\}$ and $(t,x)\in[0,1]\times\mathbb{T}^2$. In particular, this implies that $\operatorname{supp}(\widetilde{\xi}_h) \subset [0,1]\times\widetilde{\omegaup}$.
Using the change of variables $z = (s-t^a_i)(t^b_i-t^a_i)^{-1}$ under the integral sign, one can verify that
\[
\int_0^1 \widetilde{\xi}_h(s,\Phi^{\widetilde{v}}(x; 1, s)) \, \diff{s} = -\int_0^1 (\nabla \wedge h)(s,\Phi^{\widetilde{v}}(x; 1, s)) \, \diff{s}.
\]
This implies in view of Duhamel's principle that $\widetilde{w}_h(1,\cdot) = 0$.
In order to replace~$\widetilde{w}_h$ by an average-free version, we take $\widetilde{\chi} \in C^{\infty}(\mathbb{T}^2)$ with $\operatorname{supp}(\widetilde{\chi}) \subset \widetilde{\omegaup}$ and $\smallint_{\mathbb{T}^2} \widetilde{\chi}(x) \, \diff{x} = 1$. Then, we define the functions
\begin{equation*}
	\begin{gathered}
		w_h(t,x) \coloneq \widetilde{w}_h(t,x) - \widetilde{\chi}(x) \int_{\mathbb{T}^2} \widetilde{w}_h(t, z) \, \diff{z}, \\ 
		\xi_h(t,x) \coloneq \widetilde{\xi}_h(t,x) - \widetilde{\chi}(x) \int_{\mathbb{T}^2} \partial_t \widetilde{w}_h(t,z) \, \diff{z} - [(\widetilde{v}\cdot \nabla) \widetilde{\chi}](x) \int_{\mathbb{T}^2} \widetilde{w}_h(t,z) \, \diff{z},
	\end{gathered}
\end{equation*}
which satisfy
\begin{equation}\label{equation:wh}
	\begin{gathered}
		\partial_t w_h + (\widetilde{v} \cdot \nabla) w_h = \nabla \wedge h + \xi_h, \quad w_h(0, \cdot) = 0, \quad w_h(1, \cdot) = 0
	\end{gathered}
\end{equation}
and
\begin{equation}\label{equation:avg2}
	\int_{\mathbb{T}^2}w_h(t,x) \, \diff{x} = 0
\end{equation}
for all $t \in [0,1]$.

Next, following the idea of \cite{CoronMarbachSueur2020}*{Appendix A.2}, we introduce a controlled solution to a linear version of the vorticity formulation of \eqref{equation:controlled_euler} by setting
\begin{equation}\label{equation:auxilliary_vorticity_combined}
	\begin{gathered}
		w(t,x) \coloneq \sum_{i=1}^M \beta_i(t)w_i(t,x) + w_h(t,x), \quad
		\xi(t,x) \coloneq \sum_{i=1}^M \beta_i'(t)w_i(t,x) + \xi_h(t,x)
	\end{gathered}
\end{equation}
for all $(t,x) \in [0,1]\times \mathbb{T}^2$, and where $\{ \beta_i \}_{i\in\{1,\dots,M\}} \subset C^{\infty}([0,1];[0,1])$ are fixed independently of the choice of $\widetilde{v}$ such that
\begin{equation}\label{equation:beta_i}
	\beta_i(t) = \begin{cases}
		1 & \mbox{ if } 0 \leqslant t \leqslant t_i^a,\\
		0 & \mbox{ if } t \geqslant t_i^b
	\end{cases}
\end{equation}
for each $i \in \{1,\dots,M\}$.

\begin{lemma}\label{lemma:auxw}
	The functions $w$ and $\xi$ defined via \eqref{equation:auxilliary_vorticity_combined} solve the controllability problem
	\begin{equation}\label{equation:acp}
		\partial_t w + (\widetilde{v} \cdot \nabla) w = \nabla \wedge h + \xi, \quad w(0,\cdot) = \nabla \wedge v_0, \quad w(1,\cdot) = 0.
	\end{equation}
	Moreover, $w$ belongs to $C([0,1]; L^2(\mathbb{T}^2))\cap L^{\infty}([0,1]; H^2(\mathbb{T}^2))$ and has zero average, while the control satisfies $\operatorname{supp}(\xi) \subset [0,1]\times\widetilde{\omegaup}$. 
\end{lemma}
\begin{proof}
	That~$w$ solves the initial value problem in \eqref{equation:acp} with control $\xi$ can be seen by plugging \eqref{equation:auxilliary_vorticity_combined} into \eqref{equation:acp}. The regularity of $w$ is inherited from that of the solutions to the involved transport problems \eqref{equation:aux_vorticity} and \eqref{equation:aux_vorticity_h}. It follows then from the properties of $\{ \beta_i \}_{i \in \{1,\dots,M\}}$ and the definition of $w_h$ that $w(1,\cdot) = 0$. Moreover, by~\eqref{equation:avg1} and~\eqref{equation:avg2}, the function $w$ has zero average. 
	Finally, the support of $\xi$ is characterized by using~\Cref{lemma:reference_trajectory} and~\cref{equation:flushing,equation:geometric_property,equation:beta_i}. Indeed, given any $t \in [0,1]$ with $\beta_i'(t) \neq 0$ for an index $i \in \{1,\dots,M\}$, one finds that $\operatorname{supp}(w_i(t,\cdot))$, which is transported by the flow of~$\widetilde{v}$ in the sense of \eqref{equation:mcw}, is contained in~$\widetilde{\omegaup}$. 
\end{proof}

To complete the construction of $\mathscr{F}$, the function $v(t,\cdot)$ is obtained for each $t \in [0, 1]$ by solving in $\mathbb{T}^2$ the div-curl problem
\begin{gather*}
	\operatorname{div}(v(t,\cdot)) = 0, \quad \nabla \wedge v(t,\cdot) = w(t,\cdot), 
	\quad \int_{\mathbb{T}^2} v(t, x) \, \diff{x} = \overline{y}(t) + \kappa(t) \int_{\mathbb{T}^2} v_0(x) \, \diff{x},
\end{gather*}
where $\kappa \in C^{\infty}([0,1];[0,1])$ is fixed such that $\kappa(0) = 1$ and $\kappa(1) = 0$.

\begin{remark}\label{remark:streamfunctionv}
	One can construct $v$ using the stream function approach. To this end, one first solves in $\mathbb{T}^2$ the Poisson problems $\Delta \psi(t,\cdot) = -w(t,\cdot)$ with $t \in [0,1]$ and subsequently takes $v \coloneq \nabla^{\perp} \psi + A$ with $A \coloneq \overline{y} + \kappa \smallint_{\mathbb{T}^2} v_0(x) \, \diff{x}$.
\end{remark}

Finally, we define $\mathscr{F}(\widetilde{v}) \coloneq v$. In view of \Cref{remark:streamfunctionv} and classical elliptic regularity results for the Laplacian, one can infer that $\mathscr{F}(\widetilde{v}) \in V_1$.

\begin{proposition}\label{proposition:fixedpoint}
	If the data bound $\delta_0 > 0$ in \eqref{equation:smallnessconstraint} is chosen sufficiently small, the map $\mathscr{F}$ admits a unique fixed point in $X_{\delta}$.
\end{proposition}
\begin{proof}
	We show that~$\mathscr{F}$ is a contractive self-map $X_{\delta} \longrightarrow X_{\delta}$ with respect to the norm of $C([0,1];H^1(\mathbb{T}^2;\mathbb{R}^2))$ if the bound~$\delta_0 > 0$ in \eqref{equation:smallnessconstraint} is fixed sufficiently small. Note that $X_{\delta}$ is
	a closed subspace of $C([0,1];H^1(\mathbb{T}^2;\mathbb{R}^2))$ and hence a complete metric space with respect to the norm of $C([0,1];H^1(\mathbb{T}^2;\mathbb{R}^2))$.
	Thus, the conclusion follows from Banach's fixed point theorem. Below, $C > 0$ denotes a generic constant which can be different from line to line.
	
	\paragraph{Self-map property.}
	Let $\widetilde{v} \in X_{\delta}$ be arbitrary and define $v \coloneq \mathscr{F}(\widetilde{v})$. Due to the definition of $X_{\delta}$, the $L^{\infty}([0,1];H^3(\mathbb{T}^2;\mathbb{R}^2))$ norm of $\widetilde{v}$ is
	bounded by a constant depending only on $\overline y$ and on $\delta$. 
	Moreover, the associated function $\widetilde{w}_h$, arising from \eqref{equation:aux_vorticity_h} in the construction of~$\mathscr{F}$, admits the representation
	\[
	\widetilde{w}_h(t, x) = \int_0^t (\nabla\wedge h + \widetilde{\xi}_h)(\Phi^{\widetilde{v}}(x;t,s), s) \, \diff{s}
	\]
	for all $(t,x) \in [0,1]\times\mathbb{T}^2$. This implies, with the help of the continuous Sobolev embedding $H^2(\mathbb{T}^2)\hookrightarrow L^{\infty}(\mathbb{T}^2)$, the rough estimate
	\begin{equation*}
		\|\widetilde{w}_h(t, \cdot)\|_2 \leq C\|h\| _{L^2([0,1]; H^5(\mathbb{T}^2;\mathbb{R}^2))}\sup\limits_{t\in[0,1]} \|\widetilde{v}(t, \cdot)\|_3 
		\leq C\|h\| _{L^2([0,1]; H^5(\mathbb{T}^2;\mathbb{R}^2))}.
	\end{equation*}
	Therefore, in view of~\eqref{equation:auxilliary_vorticity_combined} and \Cref{remark:streamfunctionv}, it follows that
	\begin{equation*}
		\begin{aligned}
			\sup_{t \in [0,1]}\|v(t, \cdot) - \overline{y}(t)\|_{3} & \leqslant C \left( \sup_{t \in [0,1]} \| \nabla \wedge v(t,\cdot)\|_2 + \|v_0\|_3 \right) \\
			& \leqslant C \left( \sup_{t \in [0,1]} \sum_{i=1}^M \| w_i(t,\cdot)\|_2 + \|v_0\|_3 + \|h\|_{L^2([0,1]; H^5(\mathbb{T}^2;\mathbb{R}^2))} \right).
		\end{aligned}
	\end{equation*}
	The basic estimates for \eqref{equation:aux_vorticity} yield $\|w_i(t,\cdot)\|_2 \leqslant C \|v_0\|_{3}$. This shows that $v \in X_{\delta}$ whenever $\|v_0\|_3 +          \|h\|_{L^2([0,1];H^5(\mathbb{T}^2;\mathbb{R}^2))}$ is sufficiently small. Finally, we fix $\delta_0 > 0$ in~\Cref{theorem:EulerExactLocal} so that $\mathscr{F}(\widetilde{v}) \in X_{\delta}$ for each element~$\widetilde{v} \in X_{\delta}$.

	\paragraph{Contraction property.}
	We fix arbitrary elements $\widetilde{v}^a, \widetilde{v}^b \in X_{\delta}$ and follow the steps of the construction of~$\mathscr{F}$ from \Cref{subsubsection:construction_F} in order to define
	\begin{gather*}
		v^a \coloneq \mathscr{F}(\widetilde{v}^a), \quad v^b \coloneq \mathscr{F}(\widetilde{v}^b), \quad \widetilde{v} \coloneq \widetilde{v}^a - \widetilde{v}^b,\\
		v \coloneq v^a - v^b, \quad w^a \coloneq \nabla \wedge v^a, \quad w^b \coloneq \nabla \wedge v^b, \quad \zeta \coloneq w^a - w^b.
	\end{gather*}
	Moreover, for all $(t,x) \in [0,1]\times\mathbb{T}^2$ it holds $\zeta(t,x) = \sum_{i=1}^M \beta_i(t) \zeta_i(t, x)$,
	where we denote the differences $\{ \zeta_i \}_{i\in \{1,\dots, M\}} \coloneq \{w^a_i - w^b_i\}_{i\in \{1,\dots, M\}}$ for families $\{ w^a_i \}_{i\in \{1,\dots, M\}}$ and $\{ w^b_i \}_{i\in \{1,\dots, M\}}$ of solutions to problems of the form \eqref{equation:aux_vorticity} that are associated with
	$w^a$ and $w^b$, respectively. 
	
	In particular, for each $i\in \{1,\dots, M\}$, the function $\zeta_i$ solves in $[0,1]\times\mathbb{T}^2$ the problem
	\begin{gather*}
		\partial_t \zeta_i + (\widetilde{v} \cdot \nabla) \zeta_i + (\widetilde{v} \cdot \nabla) w^b_i + (\widetilde{v}^b \cdot \nabla) \zeta_i = 0, \quad
		\zeta_i(0,\cdot) = 0.
	\end{gather*}
	Furthermore, noting that $\widetilde{v}$ is average-free, for each $t \in [0,1]$ and $i \in \{1, \dots, M\}$, one has the estimate
	\begin{equation*}
		\|\zeta_i(t,\cdot)\|_{L^2}^2 \leqslant C \left( \int_0^t \|\zeta_i(s,\cdot)\|_{L^2}^2 \, \diff{s} +  \delta_0 \int_0^t \|\widetilde{v}(s,\cdot)\|_1^2 \, \diff{s} \right).
	\end{equation*}
	Hence, resorting to Gr\"onwall's inequality, one can infer that
	\[
	\sup_{t \in [0,1]} \|\mathscr{F}(\widetilde{v}^a)(t,\cdot) - \mathscr{F}(\widetilde{v}^b)(t,\cdot)\|_1^2 \leqslant C \delta_0 \sup_{t \in [0,1]} \| (\widetilde{v}^a-\widetilde{v}^b)(t,\cdot)\|_1^2.
	\]
	As a result, we can further reduce $\delta_0$ to ensure that $\mathscr{F}$ forms a contraction in the space $X_{\delta}$ with respect to the norm of $C([0,1];H^1(\mathbb{T}^2;\mathbb{R}^2))$.
\end{proof}

\subsubsection{Controlled solution}\label{subsubsection:controlledsolution}
Let~$V$ be the fixed point of~$\mathscr{F}$ provided by \Cref{proposition:fixedpoint}. From the construction of $\mathscr{F}$ in \Cref{subsubsection:construction_F}, it follows that $V$ solves in $[0,1]\times\mathbb{T}^2$ the vorticity problem
\begin{gather*}
	\partial_t W + (V \cdot \nabla) W = \nabla \wedge h + \xi, \\ \nabla \wedge V = W, \quad \operatorname{div} V = 0, \quad
	\int_{\mathbb{T}^2} V(x, \cdot) \, \diff{x} = \overline{y}(\cdot) + \kappa(\cdot) \int_{\mathbb{T}^2} v_0(x) \, \diff{x},\\
	W(0,\cdot) = \nabla \wedge v_0,
\end{gather*}
where $\xi \in C([0,1]; L^2(\mathbb{T}^2))\cap L^{\infty}([0,T]; H^2(\mathbb{T}^2))$ is obtained via \eqref{equation:auxilliary_vorticity_combined}. As the function $(t,x) \mapsto \beta_i'(t) w_i(t,x)$ from \eqref{equation:aux_vorticity} is for each $i \in \{1,\dots,M\}$ supported in $[0,1]\times\widetilde{\omegaup}$, the formula \eqref{equation:F} in the proof of
\Cref{lemma:integrating_control} provides a velocity control
\begin{gather*}
	U \in C([0,1]; L^2(\mathbb{T}^2; \mathbb{R}^2))\cap L^{\infty}([0,T]; H^2(\mathbb{T}^2; \mathbb{R}^2)),\\
	\nabla \wedge U = \xi, \quad \operatorname{supp}(U) \subset [0,1]\times\widetilde{\omegaup}.
\end{gather*}
Now, we define the function $c \in L^2([0,1]; \mathbb{R}^2)$ by setting
\begin{align*}
	c(t) &\coloneq \int_{\mathbb{T}^2} (\partial_t V(t,x) + (V(t,x) \cdot \nabla) V(t,x) - h(t,x) - U(t,x)) \, \diff{x} \\
	& = \int_{\mathbb{T}^2} (\partial_t V(t,x) - h(t,x) - U(t,x)) \, \diff{x} 
\end{align*}
for $t \in [0,1]$. Because of the assumption that $\mathbb{T}^2\setminus \omegaup$ is simply-connected, there exist $\Lambda, \Sigma \in C^{\infty}(\mathbb{T}^2;\mathbb{R}^2)$ satisfying $\nabla \wedge \Lambda = \nabla \wedge \Sigma = 0$ in $\mathbb{T}^2$ and (for instance, see Theorem~A.1 in~\cite{NersesyanRissel2024})
\begin{gather*}
	\operatorname{supp}(\Lambda) \cup \operatorname{supp}(\Sigma) \subset \omegaup, \quad
	\operatorname{span}_{\mathbb{R}} \left\{ \int_{\mathbb{T}^2} \Lambda(x) \, \diff{x}, \int_{\mathbb{T}^2} \Sigma(x) \, \diff{x} \right\} = \mathbb{R}^2.
\end{gather*}
Thus, one can write $c(t) = c_1(t) \smallint_{\mathbb{T}^2} \Lambda(x) \, \diff{x} + c_2(t) \smallint_{\mathbb{T}^2} \Sigma(x) \, \diff{x}$ for functions $c_1, c_2 \in L^2([0,1])$ and  $t\in[0,1]$.
As a result, for $u \coloneq U + c_1 \Lambda + c_2\Sigma$ it follows that
\begin{gather*}
	\nabla\wedge(\partial_t V + (V \cdot \nabla) V - h - u) = 0,\\
	\int_{\mathbb{T}^2} (\partial_t V(t,x) + (V(t,x) \cdot \nabla) V(t,x) - h(t,x) - U(t,x)- u(t,x)) \,  \diff{x} = 0.
\end{gather*}
In view of the Helmholtz decomposition theorem, this implies that $V$ satisfies \eqref{equation:controlled_euler} with the control $u$.
Since  $\kappa(0) = 1$ and $\overline{y}(0) = \overline{y}(1) = \kappa(1) = 0$, it holds $V(0,\cdot) = v_0$ and $V(1,\cdot) = 0$.

\section{Proofs of the main results}\label{sec_proofs_main}

This section is devoted to the proofs of our main results in Theorem \ref{th_rel_new} and Theorem \ref{th_con_local}.

\begin{proof}[Proof of Theorem \ref{th_rel_new}]
	Let $\tilde{v}\in C^\infty(\mathbb{T}^2; \mathbb{R}^2)$ be a relaxation enhancing field, in the sense of Definition \ref{def_enhancement}.
	We refer to Constantin et al.~\cite{constantin2008diffusion}*{Section 6} for the existence of such fields.
	Moreover, let $a>0$ be such that the solution $\phi_a$ of \eqref{prima_advectie} and  \eqref{CI_prima_prima}, with $\|f\|_{L^2(\mathbb{T}^2)}\leqslant 1$ 
	and $\int_{\mathbb{T}^2} f(x)\, \diff{x}=0$,  satisfies
	\begin{equation}\label{decadere_mare_dubla}
		\left\|\phi_a \left(\tau,\cdot\right)\right\|_{L^2(\mathbb{T}^2)} < \frac{\delta}{3}.
	\end{equation}
	Let $\tilde f\in H^{4}$ be such that $\|\tilde f\|_{L^2(\mathbb{T}^2)}\leqslant 1$, $\int_{\mathbb{T}^2} \tilde f(x)\, \diff{x}=0$, and
	\begin{equation}\label{date_aproape}
		\|f-\tilde f\|_{L^2(\mathbb{T}^2)}<\frac{2\delta}{9}.
	\end{equation}
	Denoting ${\tilde \phi}_a$ the solution of \eqref{prima_advectie} and \eqref{CI_prima_prima} with $f$ replaced by $\tilde f$
	it follows that
	\begin{equation}\label{E:1.3_new_new}
		\| {\tilde \phi}_a (\tau,\cdot) - \phi_a(\tau,\cdot)\|_{L^2(\mathbb{T}^2)} < \frac{2\delta}{9}.
	\end{equation}
	Next, for any integer $n\geq 1$, let $\theta_n \colon J_{\tau}=[0,\tau]\to [0,1]$ be a smooth non-negative function satisfying $\theta_n(0)=0$ and $\theta_n(t)=1$ for $1/n \leqslant t\leqslant \tau$. Then set
	$$
	\psi_n(t,x)\coloneq (1-\theta_n(t))v_0(x)+\theta_n(t)a\tilde{v}(x).
	$$
	The function $\psi_n$ equals to $v_0$ at $t=0$, reaches $a\tilde{v}$ at $t=1/n$, and stays there for the remaining time. Now let $\tilde{\varphi}_{\psi_n}$ be the solution of~\eqref{a_doua_advectie} and \eqref{CI_prima}, with the velocity field $\psi_n$ instead of $v$ and the initial data $\tilde f$ instead of $f$. We can then apply Theorem~\ref{T.2.1} with $\psi=\psi_n$ to find a control $u\in C^\infty(J_\tau;\mathbb{R}^m)$ such that the solution $v$ of \eqref{equation:euler_finitedimcontrol} satisfies
	\begin{equation}\label{missing}
		\|\tilde\varphi_v-\tilde{\varphi}_{\psi_n}\|_{C(J_\tau;H^2)} < \frac{\delta}{9} ,
	\end{equation}
	where $\tilde\varphi_v$ is the solution of~\eqref{a_doua_advectie} and \eqref{CI_prima}, with the initial data $\tilde f$ instead of $f$. Noting that $\| \psi_n \|_{C(J_{\tau};H_{\sigma}^{4})}\leqslant \| v_0 \|_{4}+a \| \tilde{v} \|_{4}$, we can apply Theorem~\ref{L:1} with $R=\| \tilde{f} \|_{4}+\| v_0 \|_{4}+2a\| \tilde{v} \|_{4}$ to obtain
	$$
	\|\tilde{\varphi}_{\psi_n}-\tilde{\phi}_a\|_{C(J_\tau;H^2)} \leqslant C||| \psi_n-a\tilde{v} |||_{\tau,2}^{1/2}\leqslant C\left( \frac{\| v_0-a\tilde{v} \|_2}{n} \right)^{1/2}
	$$
	for some constant $C=C(R,\tau)>0$. Taking now $n$ large enough, it follows that
	\begin{equation}\label{missing2}
		\|\tilde{\varphi}_{\psi_n}-\tilde{\phi}_a\|_{C(J_\tau;H^2)}<\frac{\delta}{9}.
	\end{equation}
	On the other hand, combining \eqref{date_aproape} and the standard energy estimate for \eqref{a_doua_advectie} and \eqref{CI_prima} (with $f-\tilde f$ instead of $f$),
	it follows that
	$$
	\|\varphi_v(\tau,\cdot)-\tilde\varphi_v(\tau,\cdot)\|_{L^2(\mathbb{T}^2)} < \frac{2\delta}{9},
	$$
	where $\varphi_v$ is the solution of~\eqref{a_doua_advectie} and \eqref{CI_prima}. 
	The last estimate, together with \eqref{E:1.3_new_new}--\eqref{missing2} implies that
	$$
	\|\varphi_v(\tau,\cdot)-\phi_a(\tau,\cdot)\|_{L^2(\mathbb{T}^2)} < \frac{2\delta}{3}.
	$$
	Combining the above estimate and \eqref{decadere_mare_dubla} yields the conclusion \eqref{pana_aici_2}, which ends the proof.
\end{proof}

The remaining part of this section is devoted to the proof of Theorem \ref{th_con_local}. An essential ingredient of this proof is an abstract theorem first established in \cite{constantin2008diffusion}. We use the following version of that result corresponding to Theorem 3.3 in~\cite{wei2021diffusion}.

\begin{theorem}\label{th_wei_abstract}
	Let $X$ be a Hilbert space, and let $\Gamma \colon \mathcal{D}(\Gamma)\to X$ be a strictly positive operator.
	Moreover, let $S \colon \mathcal{D}(S)\to X$ be skew-symmetric on $X$, with $\mathcal{D}(S)=\mathcal{D}(\Gamma^{\frac{1}{2}})$
	and $S\in{\cal L}(\mathcal{D}(\Gamma^\frac12); X)$ .
	Then the
	following two statements are equivalent:\begin{enumerate}
		\item[\rm (i)] For every $\tau,\delta > 0$  there exists $a(\tau, \delta)$ such that for any $a > a(\tau, \delta)$ and any $g\in X$
		with $\|g\|_X = 1$, the solution $\xi_a$ of 
		\begin{equation}\label{foarte_abstracta}
			\dot{\xi_a}=a S \xi_a  - \Gamma \xi_a,\qquad \xi_a(0) = g
		\end{equation}
		satisfies $\|\xi_a(\tau)\|_X < \delta$.
		\item[\rm (ii)] The operator $S$ does not have eigenvectors lying in $\mathcal{D}(\Gamma^\frac12)$.
	\end{enumerate}
\end{theorem}

We prove Theorem \ref{th_con_local} by applying Theorem \ref{th_wei_abstract}, with an appropriate choice of space $X$ and
operators $\Gamma$ and $S$. For this purpose, we introduce
the shear flow $\tilde v \colon \mathbb{T}^2\to \mathbb{R}^2$   defined by
\begin{equation}\label{def_shear}
	\tilde v(x_1,x_2)= [\alpha(x_2), 0]^{\top} \qquad \left( [x_1, x_2]^{\top} \in \mathbb{T}^2\right),
\end{equation}
where $\alpha \colon \mathbb{T}\to \mathbb{R}$.
Shear flows have been intensively considered in the context of enhanced dissipation; see, for 
instance, Bedrossian and Coti Zelati~\cite{bedrossian2017enhanced} and the references therein. For our purposes, it suffices here to note that 
we have the following result: 

\begin{lemma}\label{lema_veche}
	Let $\tilde v$ be defined by \eqref{def_shear}, and assume that $\alpha\in H^2(\mathbb{T})$ is such that $\alpha'$ has at most a finite number of zeros on $\mathbb{T}$. Moreover, let
	$X$ be the Hilbert space defined by \eqref{def_space_x}.
	Then for every $\tau,\delta>0$ there exists $a^*(\tau,\delta)$ such that for any $a\geqslant a^*(\tau,\delta)$ and  $f\in L^2(\mathbb{T}^2)$ with $\|f\|_{L^2(\mathbb{T}^2)}\leqslant 1$  the solution $\phi_a$ of \eqref{prima_advectie} and \eqref{CI_prima_prima} satisfies
	\begin{equation*}
		\|\mathrm{P}_X\phi_a (\tau,\cdot)\|_{L^2(\mathbb{T}^2)}<\delta,
	\end{equation*}
	where $\mathrm{P}_X$ is the orthogonal projection from $L^2(\mathbb{T}^2)$ onto $X$.        
\end{lemma}

\begin{proof}
	We first note that $X$ is a Hilbert space  endowed with the $L^2(\mathbb{T}^2)$ inner product. Let $\Gamma \colon \mathcal{D}(\Gamma)\to X$
	be the operator defined by  
	\begin{equation*}
		\mathcal{D}(\Gamma)=H^2(\mathbb{T}^2) \cap X,
	\end{equation*}
	\begin{equation*}
		\Gamma g= -\Delta g \qquad (g\in \mathcal{D}(\Gamma)).
	\end{equation*}
	It is easily seen that $\Gamma$ is a strictly positive operator on $X$. 
	Moreover, it is not difficult to check that $\mathcal{D}(\Gamma^\frac12)=H^1( \mathbb{T}^2)\cap X$. 
	We next set
	\begin{equation*}
		S g=-(\tilde v\cdot \nabla g) \qquad \left(g\in \mathcal{D}(S)\coloneqq \mathcal{D}(\Gamma^\frac12)\right). 
	\end{equation*}
	We clearly have that $S\in \mathcal{L}\left(\mathcal{D}(\Gamma^\frac12),X\right)$ is a skew-symmetric (unbounded) 
	operator on $X$.
	We remark that if we denote $\xi_a=\mathrm{P}_X \phi_a$, where $\phi_a$ is the solution of \eqref{prima_advectie} and \eqref{CI_prima_prima}, then 
	$\xi_a$ satisfies \eqref{foarte_abstracta}, with the operators $\Gamma$ and $S$ defined above and~$g=\mathrm{P}_X f$.
	
	In order to apply Proposition \ref{th_wei_abstract} with $X,\Gamma$, and $S$ defined above, we still have to check that $S$ has no eigenvectors in $H^1( \mathbb{T}^2)\cap X$. 
	To this aim, we remark that $Sg=\lambda g$, with $\lambda\in \mathbb{C}$ and $g\in H^1( \mathbb{T}^2)\cap X$, if and only if the equality
	\begin{equation}\label{eq_eig}
		-\alpha (x_2)\frac{\partial g}{\partial x_1}(x_1,x_2)=\lambda g(x_1,x_2)
	\end{equation}
	holds in $L^2(\mathbb{T}^2)$.
	If the above formula holds for $\lambda=0$, then $\partial_{x_1}g=0$ in $L^2(\mathbb{T}^2)$; note that we have $\alpha(x_2)\neq 0$ for all but finitely many $x_2 \in \mathbb{T}$ by the assumption on $\alpha$ and Rolle's theorem. This and the fact that $ \int_{\mathbb{T}} g(x_1,x_2)\, \diff{x}_1=0 $ for a.e.~$x_2\in \mathbb{T}$ imply that $g=0$. 
	
	Suppose now that~\eqref{eq_eig} holds with $\lambda\neq 0$. By the assumption on $\alpha$ combined with Rolle's theorem, except for a finite possibilities of $a_2 \in \mathbb{T}$, we have
	\begin{equation}\label{nedegenerare}
		\alpha(a_2)\neq 0,\qquad \alpha'(a_2)\neq 0.
	\end{equation}
	For such $a_2$, the continuity of $\alpha$ implies that there exists $\varepsilon>0$
	such that  
	\begin{equation}\label{nu_e_zero}
		\alpha(x_2)\neq 0,\qquad \alpha'(x_2)\neq 0\qquad (x_2\in (a_2-\varepsilon,a_2+\varepsilon)).
	\end{equation}
	Using \eqref{eq_eig} and \eqref{nu_e_zero}, we get that for every $x_1, a_1\in \mathbb{T}$ and a.e.~$x_2\in (a_2-\varepsilon, a_2+\varepsilon)$
	$$
	g(x_1,x_2)=\exp\left(-\frac{\lambda}{\alpha(x_2)}(x_1-a_1)\right) g(a_1,x_2) .
	$$
	The above formula, combined with \eqref{nu_e_zero} and the $2\pi$-periodicity of $g$ with respect to $x_1$, clearly implies that for every $a_2\in \mathbb{T}$ satisfying \eqref{nedegenerare} we have
	$$
	g(x_1,x_2)=0 \qquad (x_1\in \mathbb{T},\ x_2\in (a_2-\varepsilon, a_2+\varepsilon)).
	$$
	Since \eqref{nedegenerare} holds for all but a finite number of $a_2\in \mathbb{T}$, it follows that $g=0$.  
	
	In conclusion, the only function $g\in H^1( \mathbb{T}^2)\cap X$ satisfying \eqref{eq_eig} for some $\lambda\in \mathbb{C}$ is $g=0$, so that $S$ has no eigenvectors in $\mathcal{D}(\Gamma^\frac12)$.
	Thus, we can apply Proposition \ref{th_wei_abstract} to complete the proof.
\end{proof}

We are now in a position to prove our second main result.

\begin{proof}[Proof of Theorem \ref{th_con_local}] 
	Let	 $\tilde v$ be defined by \eqref{def_shear}, with $\alpha\in H^3(\mathbb{T})$ having at most a finite number of critical points. According to Lemma \ref{lema_veche} 
	for every $\tau,\delta>0$ there exists  $a>0$ such that for every $f\in L^2(\mathbb{T}^2)$, with $\|f\|_{L^2(\mathbb{T}^2)}\leqslant 1$ the solution $\phi_a=\phi_a[f]$ of \eqref{prima_advectie} and \eqref{CI_prima_prima} satisfies
	\begin{equation}\label{decadere_mare_mai}
		\left\|\mathrm{P}_X\phi_a[f] \left(\frac{\tau}{2},\cdot\right)\right\|_{L^2(\mathbb{T}^2)}<\delta.
	\end{equation}
	With $a$ and $\tilde v$ chosen as above, we know from Theorem \ref{theorem:EulerExactGlobal} that there exists 
	\[
	u \in L^{2}\left(\left[0,\frac{\tau}{2}\right]; H^2(\mathbb{T}^2; \mathbb{R}^2)\right)
	\]
	such that the corresponding solution $v$ of \eqref{equation:controlled_euler}, with $h=0$, satisfies $v\left(\frac{\tau}{2},\cdot\right) = a\tilde{v}$.
	Since the function $a\tilde{v}$ is a time-independent solution of \eqref{equation:controlled_euler} with $u=h=0$, it follows that
	if we extend $u$ to a function in $ L^{\infty}\left(\left[0,\infty\right); H^1(\mathbb{T}^2; \mathbb{R}^2)\right)$ by setting $u(t,\cdot)=0$ for $t>\frac{\tau}{2}$.
	Then the corresponding solution $v$ of \eqref{equation:controlled_euler} with $h=0$ satisfies $v\left(t,\cdot\right) = a\tilde{v}$ for every $t\geqslant \frac{\tau}{2}$.
	Noting that with the choice of $f=\varphi_v\left(\frac{\tau}{2},\cdot\right)$ we have $\left\| f \right\|_{L^2(\mathbb{T}^2)}\leqslant 1$ and $\varphi_v(\tau,\cdot)=\phi_a[f]\left(\frac{\tau}{2},\cdot\right)$, the conclusion follows from~\eqref{decadere_mare_mai}.
\end{proof}

The proof of  Proposition \ref{prop24} can be obtained by a very slight variation of the above proof. Indeed, the only change consists in remarking that $a\tilde v$, with $a$ and $\tilde v$  satisfying the assumptions in  Proposition \ref{prop24}, is a stationary solution of \eqref{prima_curgere} with $u=0$ and  $h$ chosen as in the statement of  Proposition \ref{prop24}.

\subsubsection*{Acknowledgements}

The first author has been supported by JSPS KAKENHI Grant Numbers 22K13938 and 25K07077. The fourth author has been supported by the ANR project  NumOpTes ANR-22-CE46-0005 (2023--2026).
The authors thank Sylvain Ervedoza for the careful reading of the manuscript and for suggesting the improvement of Theorem~\ref{L:1}, together with a simplification of its proof. The authors further thank the NYU-ECNU Institute of Mathematical Sciences at NYU Shanghai, where this work was initiated.

\subsubsection*{Data availability statement}
Data sharing is not applicable to this article as no datasets were generated or analyzed during the current study.

\subsubsection*{Declarations}
The authors declare that they have no conflicts of interest.

\bibliographystyle{alpha}
\bibliography{Mixing}

\end{document}